\documentclass{article}

\usepackage{arxiv}

\usepackage[utf8]{inputenc} 
\usepackage[T1]{fontenc}    
\usepackage{hyperref}       
\usepackage{url}            
\usepackage{booktabs}       
\usepackage{amsfonts}       
\usepackage{amsthm}         
\usepackage{nicefrac}       
\usepackage{microtype}      
\usepackage{lipsum}		
\usepackage{graphicx}
\usepackage{natbib}
\usepackage{doi}
\usepackage{lipsum}
\usepackage{amsfonts}
\usepackage{graphicx}
\usepackage{epstopdf,mathtools}
\usepackage{algorithmic}
\newcommand{\prox}{\mathrm{prox}}
\newcommand{\dist}{\mathrm{dist}}
\newcommand{\proj}{\mathrm{proj}}
\newcommand{\diag}{\mathrm{diag}}
\usepackage{algorithm}
\counterwithin{algorithm}{section}

\newtheorem{theorem}{Theorem}[section]
\newtheorem{lemma}[theorem]{Lemma}
\newtheorem{proposition}[theorem]{Proposition}

\newtheorem{remark}[theorem]{Remark} 
\newtheorem{definition}[theorem]{Definition}
\newtheorem{assumption}[theorem]{Assumption}

\usepackage{subcaption} 
\usepackage{amssymb} 

\title{ON THE CONVERGENCE OF DOUBLY STOCHASTIC PRIMAL-DUAL HYBRID GRADIENT METHOD}

\author{ \href{https://orcid.org/0000-0000-0000-0000}{\includegraphics[scale=0.06]{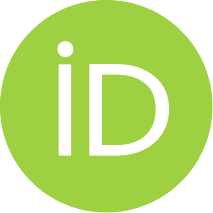}\hspace{1mm}Yiheng Xiao} \\
	Antai College of Economics and Management\\
	Shanghai Jiao Tong University\\
	\texttt{xyh051014@sjtu.edu.cn} \\
	\And
	\href{https://orcid.org/0000-0000-0000-0000}{\includegraphics[scale=0.06]{orcid.pdf}\hspace{1mm}Huikang Liu} \\
	Antai College of Economics and Management\\
	Shanghai Jiao Tong University\\
\texttt{hkl1u@sjtu.edu.cn} \\
}

\renewcommand{\shorttitle}{On the Convergence of Doubly Stochastic PDHG}

\hypersetup{
pdftitle={A template for the arxiv style},
pdfsubject={q-bio.NC, q-bio.QM},
pdfauthor={David S.~Hippocampus, Elias D.~Striatum},
pdfkeywords={First keyword, Second keyword, More},
}

\begin{document}
\maketitle

\begin{abstract}
We study a block-structured class of convex-concave saddle-point problems in which
both the primal and dual variables admit natural separable decompositions. Motivated
by large-scale applications where a full update on either side can be computationally
expensive, we propose a doubly stochastic primal--dual hybrid gradient method
(DSPDHG) that performs randomized block updates on both primal and dual variables.
The method extends classical PDHG and stochastic PDHG (SPDHG) schemes in a unified manner:
it reduces to deterministic PDHG when all blocks are selected and to one-sided
stochastic variants when only one side is randomized. For the general convex setting, we establish an $\mathcal{O}(1/K)$ ergodic convergence
rate for the expected restricted primal--dual gap under suitable blockwise step-size
conditions. We further analyze a restarted variant of DSPDHG under a  quadratic growth condition in terms of the smoothed primal-dual gap. Under this regularity
assumption, we prove linear convergence of the restarted outer iterates. Numerical evidence is provided to show that restarted DSPDHG with standard step sizes demonstrates competitive practical performance compared with PDHG, SPDHG, and their restarted variants.
\end{abstract}




\section{Introduction}

Primal--dual first-order methods have become a standard tool for solving large-scale
convex optimization problems with composite or constrained structure, particularly when accelerated with GPUs \cite{applegate2021practical,lu2023cupdlp,chen2024hpr,lu2025practical,huang2024restarted,chen2025hpr,han2024low,han2024accelerating,ding2025new,aguirre2025cuhallar,lin2025pdcs,lu2024pdot,zhang2025hot,zhang2025solving}. Among them,
the primal--dual hybrid gradient (PDHG) method \cite{chambolle2011first,chambolle2016ergodic} is particularly attractive because
each iteration only requires matrix--vector multiplications and proximal mappings,
and thus scales favorably to problems for which projections or subproblem solves are
expensive. In many applications, however, even a single full primal--dual update may
still be computationally demanding when the underlying operator is complex and the number of variables is large. 

When the primal or dual variables admit a natural block structure, a successful approach to reduce the per-iteration complexity of PDHG is  exploiting
randomization. In particular, stochastic PDHG (SPDHG) \cite{doi:10.1137/17M1134834,alacaoglu2022convergence} updates only a randomly
selected subset of dual blocks at each iteration while retaining the simplicity of PDHG, and
similar convergence guarantees under
arbitrary samplings can still be established. Existing SPDHG-type
methods, however, primarily exploit separability on the dual side. For many modern
problems, the primal variable is also block structured, and a full primal update can be
as costly as a full dual update. This naturally raises the following question: can one
randomize \emph{both} the primal and dual updates while preserving the favorable
theoretical properties of PDHG and SPDHG?

In this paper, we answer this question affirmatively. We consider the block-structured
convex optimization problem
\begin{align}\label{eq:intro-primal}
\min_{x=(x_1,\dots,x_m)}
\Biggl\{
\sum_{i=1}^n
f_i\!\left(\sum_{j=1}^m A_{ij}x_j\right)
+
\sum_{j=1}^m g_j(x_j)
\Biggr\},
\end{align}
where both the primal variable $x$ and the dual variable induced by Fenchel duality
are separable. We propose a \emph{doubly stochastic primal--dual hybrid gradient}
(DSPDHG) method, which at every iteration samples a subset of primal blocks and a
subset of dual blocks, and updates only the selected coordinates. The resulting method
can be viewed as a genuine two-sided stochastic generalization of PDHG: it reduces to
deterministic PDHG when all blocks are selected, and to singly stochastic variants when only one side is randomized.

The main technical challenge is that two independent random samplings interact through
the bilinear coupling operator. In contrast to the classical SPDHG setting, one must
control not only the stochastic dual correction but also the stochastic primal correction,
as well as the cross terms generated by the simultaneous extrapolation of both variables.
To address this difficulty, we introduce weighted Lyapunov functions adapted to the two
sampling distributions and derive one-step inequalities that connect the virtual full
updates with the actual random block updates. This allows us to establish both sublinear
and linear convergence guarantees within a unified framework.

Our main contributions are summarized as follows.

\paragraph{General convex problems}
For proper, lower semicontinuous, convex objectives and suitable blockwise step sizes,
we establish an $\mathcal{O}(1/K)$ ergodic convergence rate for the expected primal--dual
gap of DSPDHG. The result applies to general fixed-cardinality samplings on both the
primal and dual sides and extends the usual singly stochastic PDHG analysis to the
doubly stochastic regime.

\paragraph{Linear convergence under regularity}
We further show that, under a smoothed quadratic growth condition \cite{fercoq2022quadratic, lu2025practical,lu2024restarted,liu2025pdhcgscalablefirstordermethod} on a smoothed
primal--dual gap functional, a restarted version of DSPDHG converges linearly.
This condition is weaker and more flexible than global strong convexity--strong concavity,
and is tailored to the primal--dual geometry of the method.

\paragraph{A unified two-sided stochastic framework}
Our analysis makes explicit how the primal sampling probabilities, the dual sampling
probabilities, and the block interaction structure of the operator jointly determine the
admissible step sizes and convergence behavior. In particular, the deterministic PDHG
and one-sided stochastic PDHG schemes arise as special cases of our framework.

\smallskip
\textbf{Related work.}
A closely related work is the doubly stochastic primal-dual coordinate method proposed in \cite{yu2015doubly} for bilinear saddle-point problems, where linear convergence was established under strong convexity--concavity. However, the extrapolation parameter in that method depends explicitly on the strong convexity parameters. Another related method, proposed in \cite{alacaoglu2020random} for problems with a sparse coupling operator $A$, is also formulated as a primal--dual coordinate descent scheme. Its randomness, however, is induced by a single sampling mechanism and is therefore closer in spirit to a singly stochastic method. In contrast, our approach is developed for general convex saddle-point problems, employs genuinely separate samplings on the primal and dual sides, and yields both an $\mathcal{O}(1/K)$ ergodic convergence guarantee and linear convergence of the restarted scheme under a smoothed quadratic growth condition.

The rest of the paper is organized as follows. In section~\ref{sec:prelim}, we introduce
notation, the saddle-point formulation, and the general assumption used throughout
the paper. Section~\ref{sec:algorithm} presents the DSPDHG algorithm and the basic
one-step descent estimate. Section~\ref{sec:general-convex} proves the ergodic
$\mathcal{O}(1/K)$ convergence rate in the general convex setting. Section~\ref{sec:linear}
establishes linear convergence of the restarted scheme under the smoothed quadratic
growth condition. Numerical experiments are reported in section~\ref{sec:numerics}.

\section{Preliminaries}\label{sec:prelim}

\subsection{Problem formulation and saddle-point representation}

Let $\mathcal{X}_j$ $(j\in[m])$ and $\mathcal{Y}_i$ $(i\in[n])$ be finite-dimensional
Euclidean spaces, and define
\[
\mathcal{X}:=\prod_{j=1}^m \mathcal{X}_j,
\qquad
\mathcal{Y}:=\prod_{i=1}^n \mathcal{Y}_i,
\qquad
\mathcal{Z}:=\mathcal{X}\times\mathcal{Y}.
\]
For $x\in\mathcal{X}$ and $y\in\mathcal{Y}$, we write
\[
x=(x_1,\dots,x_m),\qquad y=(y_1,\dots,y_n),
\]
with $x_j\in\mathcal{X}_j$ and $y_i\in\mathcal{Y}_i$. Throughout the paper, we equip
all product spaces with the canonical inner products and their induced norms.

We consider the structured convex optimization problem
\begin{align}\label{eq:primal-problem}
\min_{x\in\mathcal{X}}
\left\{
\sum_{i=1}^n
f_i\!\left(\sum_{j=1}^m A_{ij}x_j\right)
+
\sum_{j=1}^m g_j(x_j)
\right\},
\end{align}
where each $f_i:\mathcal{Y}_i\to\mathbb{R}\cup\{+\infty\}$ and
$g_j:\mathcal{X}_j\to\mathbb{R}\cup\{+\infty\}$ is proper, lower semicontinuous, and
convex, and each $A_{ij}:\mathcal{X}_j\to\mathcal{Y}_i$ is linear. For each $i \in [n]$, let $f_i^*: \mathcal{Y}_i\to\mathbb{R}\cup\{+\infty\}$ denote the Fenchel conjugate of $f_i$. For convenience, define
\[
g(x):=\sum_{j=1}^m g_j(x_j),\qquad
f^*(y):=\sum_{i=1}^n f_i^*(y_i),
\]
and let $A:\mathcal{X}\to\mathcal{Y}$ denote the block operator
\[
(Ax)_i=\sum_{j=1}^m A_{ij}x_j,\qquad i\in[n].
\]
We also use the row and column block operators
\[
A_i:=(A_{i1},\dots,A_{im}):\mathcal{X}\to\mathcal{Y}_i,
\qquad
A^j :=
\begin{bmatrix}
A_{1j}\\ \vdots\\ A_{nj}
\end{bmatrix}
:\mathcal{X}_j\to\mathcal{Y}.
\]

By Fenchel duality, \eqref{eq:primal-problem} admits the saddle-point reformulation
\begin{align}\label{eq:saddle-problem}
\min_{x\in\mathcal{X}} \max_{y\in\mathcal{Y}}
\Phi(x,y)
:=
g(x)
+\sum_{i=1}^n \langle A_i x,y_i\rangle
-f^*(y).
\end{align}
We impose the following standard assumptions throughout the paper.

\begin{assumption}\label{ass:basic}
The functions $\{f_i\}_{i=1}^n$ and $\{g_j\}_{j=1}^m$ are proper, lower semicontinuous,
and convex. The saddle-point set $\mathcal{Z}^\star$ of \eqref{eq:saddle-problem} is
nonempty.
\end{assumption}
A pair $z^\star=(x^\star,y^\star)\in\mathcal{Z}$ is a saddle point of
\eqref{eq:saddle-problem} if and only if it satisfies the KKT system
\begin{align}\label{eq:kkt}
0\in
\begin{bmatrix}
\partial g(x^\star)+A^\top y^\star\\[1mm]
\partial f^*(y^\star)-Ax^\star
\end{bmatrix}.
\end{align}
Equivalently,
\[
-(A^{j})^\top y^\star \in \partial g_j(x_j^\star), \quad j\in [m]
\qquad
A_i x^\star \in \partial f_i^*(y_i^\star),\quad i\in[n].
\]

\subsection{Proximal mappings and block samplings}

For a proper, lower semicontinuous, convex function $h$ on a Euclidean space
$\mathcal{U}$ and a positive definite self-adjoint operator $M$, we define the proximal
mapping
\[
\prox_{M h}(u)
:=
\arg\min_{v\in\mathcal{U}}
\left\{
h(v)+\frac12\|v-u\|_{M^{-1}}^2
\right\},
\]
where $\|w\|_M^2:=\langle Mw,w\rangle$. Let $\tau_j>0$ for $j\in[m]$ and $\sigma_i>0$ for $i\in[n]$, and define the block-diagonal
operators
\[
\tau:=\diag(\tau_1 I,\dots,\tau_m I),\qquad
\sigma:=\diag(\sigma_1 I,\dots,\sigma_n I).
\]
The doubly stochastic algorithm uses two random block samplings at iteration $k+1$:
a primal sampling $S^{k+1}\subseteq[m]$ and a dual sampling $T^{k+1}\subseteq[n]$.
We write
\[
p_j:=\mathbb{P}(j\in S^{k+1}),\qquad
q_i:=\mathbb{P}(i\in T^{k+1}),
\]
and assume that these probabilities are strictly positive. Associated with the sampling
probabilities, we define
\[
P:=\diag(p_1 I,\dots,p_m I),
\qquad
Q:=\diag(q_1 I,\dots,q_n I).
\]

Given the extrapolated iterates $\bar y^k$ and $\bar x^{k+1}$, it is convenient to
introduce the corresponding \emph{virtual full updates}
\begin{align}\label{eq:virtual-full-updates}
\hat x_j^{k+1}
=
\prox_{\tau_j g_j}\!\bigl(x_j^k-\tau_j (A^j)^\top \bar y^k\bigr),
\qquad
\hat y_i^{k+1}
=
\prox_{\sigma_i f_i^*}\!\bigl(y_i^k+\sigma_i A_i\bar x^{k+1}\bigr),
\end{align}
for all $j\in[m]$ and $i\in[n]$. 
The actual algorithm updates only the coordinates in $S^{k+1}$ and $T^{k+1}$; the
operators $P^{-1}$ and $Q^{-1}$ then appear naturally in the extrapolation formulas
and in the unbiased representation of the stochastic iterates.

\subsection{Weighted metrics and primal--dual gap functionals}

The analysis is carried out in the weighted norm
\begin{align}\label{eq:Vnorm}
\|z\|_V^2
:=
\|x\|_{\tau^{-1}P^{-1}}^2+\|y\|_{\sigma^{-1}Q^{-1}}^2,
\qquad z=(x,y)\in\mathcal{Z}.
\end{align}
We also use the primal--dual gap kernel
\begin{align}\label{eq:H-def}
H(\bar x,\bar y,x, y)=H(\bar z,z)
:=
\Phi(\bar x,y)-\Phi(x,\bar y),
\end{align}
which is nonnegative when $z$ is a saddle point. For a set
$B\subseteq\mathcal{Z}$, the restricted primal--dual gap is defined by
\begin{align}\label{eq:gap-B}
G_B(\bar x,\bar y)
:=
\sup_{z=(x,y)\in B} H(\bar x,\bar y,x, y).
\end{align}
This quantity will be used to state the ergodic $\mathcal{O}(1/K)$ convergence result
in the general convex setting.

To capture the linear convergence regime, we further define the smoothed gap
functional
\begin{align}\label{eq:smoothed-gap}
G_\mu(\bar z,\dot z)
:=
\sup_{z\in\mathcal{Z}}
\left\{
H(\bar x,\bar y,x, y)-\frac{\mu}{2}\|z-\dot z\|_V^2
\right\},
\qquad
\bar z=(\bar x,\bar y),\ \dot z\in\mathcal{Z},
\end{align}
where $\mu>0$ is a smoothing parameter.

\section{The DSPDHG Algorithm}\label{sec:algorithm}

We now present the doubly stochastic primal--dual hybrid gradient method. In contrast
to classical SPDHG, which randomizes only the dual update, the proposed method
randomizes \emph{both} the primal and dual block updates. This is particularly useful
when both sides of the saddle-point formulation admit a block decomposition and full
updates on either side are expensive.

The method maintains both the current iterates $(x^k,y^k)$ and the extrapolated
variables $(\bar x^{k},\bar y^{k})$. Given $(x^k, \bar y^k)$, the primal update is first
performed on the sampled primal coordinates, followed by an extrapolation in the primal
variable. Then the dual update is carried out on the sampled dual coordinates, followed
by an extrapolation in the dual variable. 

\begin{algorithm}[h]
\caption{Doubly stochastic primal--dual hybrid gradient (DSPDHG)}
\label{alg:DSPDHG}
\begin{algorithmic}[1]
\STATE \textbf{Input:} stepsizes $\{\tau_j\}_{j=1}^m$, $\{\sigma_i\}_{i=1}^n$, initial points
$x^0\in\mathcal X$, $y^0\in\mathcal Y$, and set $\bar y^0=y^0$
\FOR{$k=0,1,2,\dots$}
    \STATE Draw a primal block set $S^{k+1}\subseteq[m]$
    \STATE
    Update  $x_j^{k+1}=
        \begin{cases}
        \prox_{\tau_j g_j}\!\bigl(x_j^k-\tau_j (A^j)^\top \bar y^k\bigr), & j\in S^{k+1},\\
        x_j^k, & j\notin S^{k+1}
        \end{cases}$
    \STATE
    Set $\bar x^{k+1}=x^{k+1}+ (P^{-1} - I)(x^{k+1}-x^k)$
    \STATE Draw a dual block set $T^{k+1}\subseteq[n]$
\STATE
        Update $y_i^{k+1}=
        \begin{cases}
        \prox_{\sigma_i f_i^*}\!\bigl(y_i^k+\sigma_i A_i\bar x^{k+1}\bigr), & i\in T^{k+1},\\
        y_i^k, & i\notin T^{k+1}
        \end{cases}$
    \STATE
    Set $\bar y^{k+1}=y^{k+1}+Q^{-1}(y^{k+1}-y^k)$
    \STATE
\ENDFOR
\end{algorithmic}
\end{algorithm}

The intuition behind these two extrapolations is as follows.
The full updates $\hat x^{k+1}$ and $\hat y^{k+1}$ given by \eqref{eq:virtual-full-updates} are not computed in practice, but it is easy to see that
\[
\begin{aligned}
    \hat x^{k+1} & =P^{-1}\mathbb E[x^{k+1}\mid \mathcal{F}_k]+(I-P^{-1})x^k, \\
    \hat y^{k+1} &=Q^{-1}\mathbb E[y^{k+1} \mid \mathcal{F}_{k+\frac{1}{2}}]+(I-Q^{-1})y^k.
\end{aligned}
\]
where $\mathcal F_k$ and $\mathcal F_{k+\frac{1}{2}}$ are the sigma-fields generated by $\{S^1,\cdots, S^k,T^1,\cdots,T^k\}$ and  $\{S^1,\cdots, S^{k+1},T^1,\cdots,T^k\}$, respectively. 
Therefore, the actual stochastic updates can be
represented as unbiased blockwise realizations of \eqref{eq:virtual-full-updates} after
rescaling by $P^{-1}$ and $Q^{-1}$, i.e.,
\begin{align}
    \mathbb{E}[\bar x^{k+1} \mid \mathcal{F}_k] = \hat x^{k+1}, \quad \text{and} \quad \mathbb{E}[\bar y^{k+1}\mid \mathcal{F}_{k+\frac{1}{2}}] = \hat y^{k+1} + \mathbb{E}[y^{k+1} - y^k\mid\mathcal{F}_{k+\frac{1}{2}}]
    \label{eq 3.2}
\end{align}
where the last term $\mathbb{E}[y^{k+1} - y^k\mid\mathcal{F}_{k+\frac{1}{2}}]$ works as the extrapolation in the dual variables of the standard PDHG method.

Our convergence analysis follows the standard PDHG philosophy
of deriving a one-step estimate from the proximal optimality conditions, but the doubly
stochastic setting requires additional correction terms to account for the mismatch
between virtual full updates and realized random block updates.

\subsection{A one-step estimate}

Recall the primal--dual gap kernel
\[
H(\bar x,\bar y,x,y):=\Phi(\bar x,y)-\Phi(x,\bar y).
\]
The next lemma is a one-step estimate for the virtual full updates. It is obtained by
applying the proximal optimality condition separately to the primal and dual virtual
updates.

\begin{lemma}\label{lem:virtual-onestep}
Suppose Assumption~\ref{ass:basic} holds. For any $k\ge 0$ and any
$(x,y)\in\mathcal Z$, the virtual full updates satisfy
\begin{align}
    \begin{aligned}
        H(\hat x^{k+1},\hat y^{k+1},x,y)
        \le\;&
\langle A\hat x^{k+1},y\rangle
-\langle Ax,\hat y^{k+1}\rangle\\
&+\langle A(x-\hat x^{k+1}),\bar y^k\rangle
+\langle A\bar x^{k+1},\hat y^{k+1}-y\rangle\\
&+\frac12\|x^k-x\|_{\tau^{-1}}^2
-\frac12\|\hat x^{k+1}-x^k\|_{\tau^{-1}}^2
-\frac12\|\hat x^{k+1}-x\|_{\tau^{-1}}^2\\
&+\frac12\|y^k-y\|_{\sigma^{-1}}^2
-\frac12\|\hat y^{k+1}-y^k\|_{\sigma^{-1}}^2
-\frac12\|\hat y^{k+1}-y\|_{\sigma^{-1}}^2 .
    \end{aligned}
    \label{eq:virtual-one-step}
\end{align}
\end{lemma}

\begin{proof}
We use the standard proximal inequality: if
\(\hat u=\prox_{\mu h}(u)\), then
\[
h(v)-h(\hat u)
\ge
\langle \mu^{-1}(u-\hat u),v-\hat u\rangle ,
\qquad \forall v .
\]
Applying this inequality to
\[
\hat x_j^{k+1}
=
\prox_{\tau_j g_j}
\bigl(x_j^k-\tau_j(A^j)^\top \bar y^k\bigr),
\]
and using the identity
\[
2\langle a-b,c-b\rangle
=
\|a-b\|^2+\|c-b\|^2-\|a-c\|^2,
\]
we obtain, for each \(j\in[m]\),
\begin{align*}
g_j(\hat x_j^{k+1})-g_j(x_j)
\le\;&
\langle A^j(x_j-\hat x_j^{k+1}),\bar y^k\rangle  \\
&+\frac{1}{2\tau_j}\|x_j^k-x_j\|^2
-\frac{1}{2\tau_j}\|\hat x_j^{k+1}-x_j^k\|^2
-\frac{1}{2\tau_j}\|\hat x_j^{k+1}-x_j\|^2 .
\end{align*}
Summing over \(j\) gives
\begin{align}
g(\hat x^{k+1})-g(x)
\le\;&
\langle A(x-\hat x^{k+1}),\bar y^k\rangle
+\frac12\|x^k-x\|_{\tau^{-1}}^2
\nonumber\\
&-\frac12\|\hat x^{k+1}-x^k\|_{\tau^{-1}}^2
-\frac12\|\hat x^{k+1}-x\|_{\tau^{-1}}^2 .
\label{eq:x-prox-bound}
\end{align}

Similarly, applying the same proximal inequality to
\[
\hat y_i^{k+1}
=
\prox_{\sigma_i f_i^*}
\bigl(y_i^k+\sigma_i A_i\bar x^{k+1}\bigr),
\]
we obtain, after summing over \(i\in[n]\),
\begin{align}
f^*(\hat y^{k+1})-f^*(y)
\le\;&
\langle A\bar x^{k+1},\hat y^{k+1}-y\rangle
+\frac12\|y^k-y\|_{\sigma^{-1}}^2
\nonumber\\
&-\frac12\|\hat y^{k+1}-y^k\|_{\sigma^{-1}}^2
-\frac12\|\hat y^{k+1}-y\|_{\sigma^{-1}}^2 .
\label{eq:y-prox-bound}
\end{align}

Finally, by the definition of the gap kernel, we have
\[
H(\hat x^{k+1},\hat y^{k+1},x,y)
=
\langle A\hat x^{k+1},y\rangle
-\langle Ax,\hat y^{k+1}\rangle
+g(\hat x^{k+1})-g(x)
+f^*(\hat y^{k+1})-f^*(y).
\]
Substituting \eqref{eq:x-prox-bound} and \eqref{eq:y-prox-bound} into the above
identity yields \eqref{eq:virtual-one-step}.
\end{proof}

The main difficulty is to convert \eqref{eq:virtual-one-step}, which is written in
terms of the virtual full updates, into an inequality involving the actual stochastic
iterates. For fixed $(x,y)\in\mathcal Z$, define
\begin{align}\label{eq:Gamma-k}
\begin{aligned}
\Gamma_k
:=\;&
\langle A(\bar x^{k+1}-x),Q^{-1}(y^{k+1}+y^{k-1}-2y^k)\rangle\\
&+\frac12\|x^k-x\|_{\tau^{-1}P^{-1}}^2
-\frac12\|x^{k+1}-x^k\|_{\tau^{-1}P^{-1}}^2
-\frac12\|x^{k+1}-x\|_{\tau^{-1}P^{-1}}^2\\
&+\frac12\|y^k-y\|_{\sigma^{-1}Q^{-1}}^2
-\frac12\|y^{k+1}-y^k\|_{\sigma^{-1}Q^{-1}}^2
-\frac12\|y^{k+1}-y\|_{\sigma^{-1}Q^{-1}}^2 .
\end{aligned}
\end{align}
Here and below we use the convention $y^{-1}=y^0$.

\begin{lemma}\label{lem:Gamma-bound}
Suppose Assumption~\ref{ass:basic} holds. For all $k\ge 0$ and all
$(x,y)\in\mathcal Z$,
\begin{align}\label{eq:Gamma-bound}
\mathbb E\!\left[
H(\hat x^{k+1},\hat y^{k+1},x,y)
\right]
\le
\mathbb E[\Gamma_k].
\end{align}
\end{lemma}

\begin{proof}
Recall that $\mathcal F_k$ and $\mathcal F_{k+\frac{1}{2}}$ are the sigma-fields generated by $S^1,\cdots, S^k,$ $T^1,\cdots,T^k$ and  $S^1,\cdots, S^{k+1},T^1,\cdots,T^k$, respectively. By the definitions of the stochastic updates and extrapolations,
\[
\bar x^{k+1}=x^k+P^{-1}(x^{k+1}-x^k),
\qquad
\bar y^k=y^k+Q^{-1}(y^k-y^{k-1}),
\]
and hence
\begin{align}\label{eq:unbiased-x}
\mathbb E[\bar x^{k+1}\mid \mathcal F_k]=\hat x^{k+1}.
\end{align}
Similarly, since
\[
\mathbb E[y^{k+1}-y^k\mid \mathcal F_{k+\frac{1}{2}}]
=
Q(\hat y^{k+1}-y^k),
\]
we have
\begin{align}\label{eq:unbiased-y}
\hat y^{k+1}
=
y^k+Q^{-1}
\mathbb E[y^{k+1}-y^k\mid \mathcal F_{k+\frac{1}{2}}].
\end{align}

We first compare the quadratic terms. Define $\hat{\mathcal{Q}}_k$ and $\mathcal{Q}_k$ to be the quadratic terms in \eqref{eq:virtual-one-step} and \eqref{eq:Gamma-k} respectively.
\begin{align}
    \begin{aligned}
        \hat{\mathcal{Q}}_k: 
        =& \frac12\|x^k-x\|_{\tau^{-1}}^2
-\frac12\|\hat x^{k+1}-x^k\|_{\tau^{-1}}^2
-\frac12\|\hat x^{k+1}-x\|_{\tau^{-1}}^2\\
&+\frac12\|y^k-y\|_{\sigma^{-1}}^2
-\frac12\|\hat y^{k+1}-y^k\|_{\sigma^{-1}}^2
-\frac12\|\hat y^{k+1}-y\|_{\sigma^{-1}}^2,\\
\mathcal{Q}_k:
=&\frac12\|x^k-x\|_{\tau^{-1}P^{-1}}^2
-\frac12\|x^{k+1}-x^k\|_{\tau^{-1}P^{-1}}^2
-\frac12\|x^{k+1}-x\|_{\tau^{-1}P^{-1}}^2\\
&+\frac12\|y^k-y\|_{\sigma^{-1}Q^{-1}}^2
-\frac12\|y^{k+1}-y^k\|_{\sigma^{-1}Q^{-1}}^2
-\frac12\|y^{k+1}-y\|_{\sigma^{-1}Q^{-1}}^2
    \end{aligned}
    \label{Q_k_def}
\end{align}
For the primal part, the blockwise update rule
implies
\begin{align}
&\mathbb E\!\left[
\frac12\|x^k-x\|_{\tau^{-1}P^{-1}}^2
-\frac12\|x^{k+1}-x^k\|_{\tau^{-1}P^{-1}}^2
-\frac12\|x^{k+1}-x\|_{\tau^{-1}P^{-1}}^2
\,\middle|\,\mathcal F_k
\right]
\nonumber\\
&\qquad =
\frac12\|x^k-x\|_{\tau^{-1}}^2
-\frac12\|\hat x^{k+1}-x^k\|_{\tau^{-1}}^2
-\frac12\|\hat x^{k+1}-x\|_{\tau^{-1}}^2 .
\label{eq:primal-quad-unbiased}
\end{align}
Indeed, this is because for any $\mathcal F_k$ measurable random vector $x$,
\begin{align*}
    \forall j \in [m],\quad \mathbb E\left[\|x_j^{k+1}-x_j\|^2\mid \mathcal{F}_k\right]=p_j\|\hat x_j^{k+1}-x_j\|^2+(1-p_j)\|x_j^k-x_j\|^2.
\end{align*}
The same argument applied conditionally on $\mathcal F_{k+\frac{1}{2}}$ gives
\begin{align}
&\mathbb E\!\left[
\frac12\|y^k-y\|_{\sigma^{-1}Q^{-1}}^2
-\frac12\|y^{k+1}-y^k\|_{\sigma^{-1}Q^{-1}}^2
-\frac12\|y^{k+1}-y\|_{\sigma^{-1}Q^{-1}}^2
\,\middle|\,\mathcal F_{k+\frac{1}{2}}
\right]
\nonumber\\
&\qquad =
\frac12\|y^k-y\|_{\sigma^{-1}}^2
-\frac12\|\hat y^{k+1}-y^k\|_{\sigma^{-1}}^2
-\frac12\|\hat y^{k+1}-y\|_{\sigma^{-1}}^2 .
\label{eq:dual-quad-unbiased}
\end{align}
Combining \eqref{eq:primal-quad-unbiased} and \eqref{eq:dual-quad-unbiased}, we have
\begin{align}
    \mathbb E[\hat{\mathcal{Q}}_k]=\mathbb E[\mathcal{Q}_k]
    \label{eq q_unbiased}
\end{align}

It remains to rewrite the bilinear terms in \eqref{eq:virtual-one-step}. Denote
\[
B_k :=
\langle A\hat x^{k+1},y\rangle
-\langle Ax,\hat y^{k+1}\rangle
+\langle A(x-\hat x^{k+1}),\bar y^k\rangle
+\langle A\bar x^{k+1},\hat y^{k+1}-y\rangle .
\]
Using \eqref{eq:unbiased-x} and the fact that $y$ and $\bar y^k$ are
$\mathcal F_k$-measurable, we obtain
\begin{align*}
\mathbb E[B_k]
&=
\mathbb E\!\left[
\langle A\bar x^{k+1},y\rangle
-\langle Ax,\hat y^{k+1}\rangle
+\langle A(x-\bar x^{k+1}),\bar y^k\rangle
+\langle A\bar x^{k+1},\hat y^{k+1}-y\rangle
\right]\\
&=
\mathbb E\!\left[
\langle A(\bar x^{k+1}-x),\hat y^{k+1}-\bar y^k\rangle
\right].
\end{align*}
By \eqref{eq:unbiased-y} and the definition of $\bar y^k$,
\[
\hat y^{k+1}-\bar y^k
=
Q^{-1}\mathbb E[y^{k+1}-y^k\mid \mathcal F_{k+\frac{1}{2}}]
-
Q^{-1}(y^k-y^{k-1}).
\]
Since $\bar x^{k+1}$ is $\mathcal F_{k+\frac{1}{2}}$-measurable, it follows that
\begin{align}
\mathbb E[B_k]
&=
\mathbb E\!\left[
\langle A(\bar x^{k+1}-x),
Q^{-1}(y^{k+1}+y^{k-1}-2y^k)\rangle
\right].
\label{eq:bilinear-unbiased}
\end{align}

Taking expectations in \eqref{eq:virtual-one-step} and substituting
\eqref{eq q_unbiased}, \eqref{eq:bilinear-unbiased}, we obtain exactly
\[
\mathbb E\!\left[
H(\hat x^{k+1},\hat y^{k+1},x,y)
\right]
\le
\mathbb E[\Gamma_k],
\]
which proves the claim.
\end{proof}

\subsection{Step-size conditions}

The doubly stochastic setting generates coupling terms involving both the sampled
primal and dual coordinates. To control these terms, we introduce the block operator
constants
\begin{align}
\Lambda_{r,s}
&\coloneqq
\max\bigl\{\|A_{IJ}\|:\ I\subseteq[n],\ J\subseteq[m],\ |I|=r,\ |J|=s\bigr\},\label{eq:Lambda_rs}\\
\Lambda_r
&\coloneqq
\max\bigl\{\|A_I\|:\ I\subseteq[n],\ |I|=r\bigr\},
\end{align}
where \(A_{IJ}\) denotes the restriction of \(A\) to the row block set \(I\) and
the column block set \(J\), and \(A_I\) denotes the restriction of \(A\) to the row
block set \(I\). The norms above are the induced operator norms under the canonical
Euclidean product norms.

We focus on fixed-cardinality samplings satisfying
\[
|S^{k+1}|=s,\qquad |T^{k+1}|=r,\qquad \forall k\ge 0.
\]
For uniform fixed-cardinality samplings, the corresponding marginal probabilities are
\[
p_j=p=\frac{s}{m},\qquad q_i=q=\frac{r}{n}.
\]

\begin{assumption}\label{ass:stepsizes-general}
We use uniform fixed-cardinality samplings and equal step sizes, namely
\[
p_j=p,\qquad q_i=q,\qquad \tau_j=\tau,\qquad \sigma_i=\sigma .
\]
Besides, we assume that the sampling sets $S^k, T^k, k \ge 1$ are all independent. Moreover, the step sizes and sampling probabilities satisfy
\begin{align}
    \Lambda_{r,s}\,
p^{-1/2}q^{-1/2}\tau^{1/2}\sigma^{1/2}
& =
\gamma_1^2  
<
\frac12,  \label{eq:gamma1}\\
\Lambda_r\,
p^{1/2}q^{-1/2}\tau^{1/2}\sigma^{1/2}
&=
\gamma_2^2
\le
\frac12 . \label{eq:gamma2}
\end{align}
\end{assumption}

\begin{remark}\label{rem:stepsizes-general}
The assumptions of uniform sampling and equal step sizes are imposed only to simplify
the notation and the proof of Lemma~\ref{lem:V-positivity}. The analysis can be
extended to nonuniform samplings and block-dependent step sizes by using an expected
separable overapproximation (ESO) inequality, as in
\cite{doi:10.1137/17M1134834}. The remaining parts of the argument do not rely on
uniform sampling or equal step sizes.
\end{remark}


\begin{remark}
The constant $\frac{1}{2}$ in our step-size condition is more conservative than the constant $1$ typically used in deterministic PDHG. This additional restriction arises from the stochastic updates on both the primal and dual sides. In our numerical experiments, however, we observed that the larger constant $1$ also works well in practice.
\end{remark}

Assumption~\ref{ass:stepsizes-general} is a two-sided analogue of the standard
SPDHG step-size restriction. The first condition controls the interaction between
the stochastic primal and dual increments, while the second ensures positivity of
the time-dependent Lyapunov functional used below.

\section{General Convex Convergence}\label{sec:general-convex}

We now establish the basic convergence guarantees of DSPDHG in the general convex
setting. Our goal is to prove an $\mathcal O(1/K)$ ergodic convergence rate for the
expected restricted primal--dual gap.

\subsection{Lyapunov functions}

To telescope the one-step estimate, we introduce two weighted quadratic functionals.
For $z=(x,y)$, define
\begin{align}
V(z)
=
\frac14\|x\|_{\tau^{-1}P^{-1}}^2
+\frac14\|y\|_{\sigma^{-1}Q^{-1}}^2
+\langle AP^{-1}x,Q^{-1}y\rangle, \label{eq:V-def}\\
V_k(z)
=
\frac12\|x\|_{\tau^{-1}P^{-1}}^2
+\frac12\|y\|_{\sigma^{-1}Q^{-1}}^2
+\frac14\|y^k-y^{k-1}\|_{\sigma^{-1}Q^{-1}}^2
-\langle Ax,Q^{-1}(y^k-y^{k-1})\rangle . \label{eq:Vk-def}
\end{align}
Here and below, we use the convention $y^{-1}=y^0$. Under
Assumption~\ref{ass:stepsizes-general}, these functionals satisfy the coercivity and
descent properties needed for the telescoping argument.

\begin{lemma}\label{lem:V-positivity}
Suppose Assumption~\ref{ass:stepsizes-general} holds. Then for all $k\ge 0$:
\begin{align}
&\hspace{1.5em}V(x^{k+1}-x^k,y^k-y^{k-1})
\ge
\frac{1-2\gamma_1^2}{4}
\left(
\|x^{k+1}-x^k\|_{\tau^{-1}P^{-1}}^2
+
\|y^k-y^{k-1}\|_{\sigma^{-1}Q^{-1}}^2
\right), \label{eq:V-coercive}
\\
&\hspace{1.5em}V_k(z)
\ge
\frac14\|x\|_{\tau^{-1}P^{-1}}^2+\frac12\|y\|_{\sigma^{-1}Q^{-1}}^2,
\label{eq:Vk-positive}
\\
&\hspace{1.5em}V_k(x^k-x,y^k-y)-V_{k+1}(x^{k+1}-x,y^{k+1}-y)
\ge
V(x^{k+1}-x^k,y^k-y^{k-1})+\Gamma_k.
\label{eq:Vk-descent}
\end{align}
\end{lemma}

\begin{proof}
    First, we use \eqref{eq:gamma1} to prove \eqref{eq:V-coercive}. By definition, we have
    \begin{align}
        \begin{aligned}
             V(x^{k+1}-x^k,y^k-y^{k-1})=& \frac{1}{4}\|x^{k+1}-x^k\|_{\tau^{-1}P^{-1}}^2+\frac{1}{4}\|y^k-y^{k-1}\|_{\sigma^{-1}Q^{-1}}^2\\ &+\langle AP^{-1} (x^{k+1}-x^k), Q^{-1}(y^k-y^{k-1})\rangle.
        \end{aligned}
        \label{q 24}
    \end{align}
Let $J=S^{k+1},I=T^k$. Only for $j \in J$, $x^{k+1}_j-x_j^k$ is nonzero and only for $i \in I$, $y_i^k-y_i^{k-1}$ is nonzero. Therefore,
\begin{align}
    \begin{aligned}
        &\left|\langle AP^{-1} (x^{k+1}-x^k), Q^{-1}(y^k-y^{k-1})\rangle\right|=\left| \sum_{i=1}^n \sum_{j=1}^m \frac{1}{pq} \langle A_{ij}(x_j^{k+1}-x_j^k), y_i^k-y_i^{k-1} \rangle  \right|\\
         =&\left|\sum_{i \in I} \sum_{j \in J} \tau^\frac{1}{2}\sigma^{\frac{1}{2}}p^{-\frac{1}{2}}q^{-\frac{1}{2}} \langle A_{ij}(x_j^{k+1}-x_j^k)\tau^{-\frac{1}{2}}p^{-\frac{1}{2}}, (y_i^k-y_i^{k-1})\sigma^{-\frac{1}{2}}q^{-\frac{1}{2}} \rangle\right|\\
        =&\left(  p^{-\frac{1}{2}}q^{-\frac{1}{2}}\tau^{\frac{1}{2}}\sigma^{\frac{1}{2}}\right)\cdot \left| \langle A_{IJ}p^{-\frac{1}{2}}\tau^{-\frac{1}{2}}(x_J^{k+1}-x_J^k),\sigma^{-\frac{1}{2}}q^{-\frac{1}{2}}(y_I^{k}-y_I^{k-1}) \rangle   \right|\\
       \leq &\left( p^{-\frac{1}{2}}q^{-\frac{1}{2}}\tau^{\frac{1}{2}}\sigma^{\frac{1}{2}}\right)\cdot \|A_{IJ}\|\cdot\|p^{-\frac{1}{2}}\tau^{-\frac{1}{2}}(x_J^{k+1}-x_J^k)\| \cdot \|\sigma^{-\frac{1}{2}}q^{-\frac{1}{2}}(y_I^{k}-y_I^{k-1})\|\\
        \leq &\frac{\gamma_1^2}{2} \left(\|p^{-\frac{1}{2}}\tau^{-\frac{1}{2}}(x_J^{k+1}-x_J^k)\|^2+\|\sigma^{-\frac{1}{2}}q^{-\frac{1}{2}}(y_I^{k}-y_I^{k-1})\|^2 \right)\\
        = &\frac{\gamma_1^2}{2} \left(\|x^{k+1}-x^k\|_{\tau^{-1}P^{-1}}^2 +\|y^k-y^{k-1}\|_{\sigma^{-1}Q^{-1}}^2 \right).
    \end{aligned}
    \label{q 25}
\end{align}
Combining \eqref{q 24} and \eqref{q 25}, we get \eqref{eq:V-coercive}.

Second, we use \eqref{eq:gamma2} to prove \eqref{eq:Vk-positive}. It is obvious that we only need to show
\begin{align}
    \frac{1}{4}\|x\|_{\tau^{-1}P^{-1}}^2+\frac{1}{4}\|y^k-y^{k-1}\|_{\sigma^{-1}Q^{-1}}^2-\langle Ax,Q^{-1}(y^k-y^{k-1})\rangle \ge 0
    \label{q 26}
\end{align}
Similar to the previous computation, we have
\begin{align}
        \left|\langle Ax, Q^{-1}(y^k-y^{k-1})  \rangle\right| &= \left|\sum_{i \in I}\sum_{j=1}^m \sigma^{\frac{1}{2}}\tau^{\frac{1}{2}}p^\frac{1}{2}q^{-\frac{1}{2}} \langle A_{ij}x_j \tau^{-\frac{1}{2}}p^{-\frac{1}{2}}, (y_i^k-y_i^{k-1})\sigma^{-\frac{1}{2}}q^{-\frac{1}{2}} \rangle\right| \nonumber\\
        & = \left(p^{\frac{1}{2}}q^{-\frac{1}{2}}\tau^{\frac{1}{2}}\sigma^{\frac{1}{2}}\right)\cdot \left| \langle  p^{-\frac{1}{2}}\tau^{-\frac{1}{2}}A_Ix,\sigma^{-\frac{1}{2}} q^{-\frac{1}{2}}(y_I^k-y_I^{k-1} )\rangle   \right| \nonumber\\
        & \leq \frac{\gamma_2^2}{2} \left(\|p^{-\frac{1}{2}}\tau^{-\frac{1}{2}}x\|^2+\|\sigma^{-\frac{1}{2}} q^{-\frac{1}{2}}(y_I^k-y_I^{k-1} )\|^2 \right)\\
        &\leq \frac{1}{4}\|x\|_{\tau^{-1}P^{-1}}^2+\frac{1}{4}\|y^k-y^{k-1}\|_{\sigma^{-1}Q^{-1}}^2 \nonumber.
\end{align}
This inequality directly leads to \eqref{q 26}.

Finally, we use \eqref{eq:gamma1} to prove \eqref{eq:Vk-descent}. We can compute
\begin{align}
    \begin{aligned}
         V_k&(x^k-x,y^k-y)-V_{k+1}(x^{k+1}-x,y^{k+1}-y)-V(x^{k+1}-x^{k},y^k-y^{k-1})-\Gamma_k\\
         =&\frac{1}{4}\|x^{k+1}-x^k\|_{\tau^{-1}P^{-1}}^2+\frac{1}{4}\|y^{k+1}-y^{k}\|_{\sigma^{-1}Q^{-1}}^2\\
         &+\langle A(I-P^{-1})(x^{k+1}-x^k),Q^{-1}(y^{k+1}-y^k)\rangle.
    \end{aligned}
\end{align}
Following the  same procedure as in \eqref{q 25}, and letting $\gamma_3^2=\Lambda_{r,s} \cdot( p^{-\frac{1}{2}}-p^\frac{1}{2})q^{-\frac{1}{2}}\tau^{\frac{1}{2}}\sigma^{\frac{1}{2}}$, we have 
\begin{align}
    \begin{aligned}
        &\left|\langle A(I-P^{-1})(x^{k+1}-x^k),Q^{-1}(y^{k+1}-y^k)\rangle\right| \\
        &\leq \frac{\gamma_3^2}{2}\left(\|x^{k+1}-x^k\|^2_{\tau^{-1}P^{-1}}+\|y^{k+1}-y^k\|_{\sigma^{-1}Q^{-1}}^2\right)\\
        & \leq \frac{1}{4}\left(\|x^{k+1}-x^k\|^2_{\tau^{-1}P^{-1}}+\|y^{k+1}-y^k\|_{\sigma^{-1}Q^{-1}}^2\right),
    \end{aligned}
\end{align}
where the last inequality follows from $\gamma_3^2 \leq \gamma_1^2 \leq \frac{1}{2}$. This inequality directly leads to \eqref{eq:Vk-descent}.
\end{proof}

Combining Lemmas~\ref{lem:Gamma-bound} and \ref{lem:V-positivity} yields the
fundamental recursive estimate.

\begin{lemma}\label{lem:fundamental-recursion}
Under Assumption~\ref{ass:basic} and Assumption~\ref{ass:stepsizes-general}, for every $k\ge 0$ and $(x,y)\in\mathcal Z$,
\begin{align}
\mathbb E\!\left[H(\hat x^{k+1},\hat y^{k+1},x,y)\right]
\le\;&
\mathbb E\!\left[V_k(x^k-x,y^k-y)\right]
-
\mathbb E\!\left[V_{k+1}(x^{k+1}-x,y^{k+1}-y)\right]
\nonumber\\
&-
\mathbb E\!\left[V(x^{k+1}-x^k,y^k-y^{k-1})\right].
\label{f 30}
\end{align}
In particular, if $z^\star=(x^\star,y^\star)\in\mathcal Z^\star$ is a saddle point, then
\begin{align}
    \begin{aligned}
           &\mathbb{E}\left[V_k(x^k-x^\star,y^k-y^\star)\right] \leq  d_0, \quad \forall k \ge 0,\\
    &\sum_{k=0}^{\infty} \mathbb{E}\left[V(x^{k+1}-x^{k},y^k-y^{k-1})\right] \leq d_0,
    \end{aligned}
    \label{eq 26}
\end{align}
where $d_0=\mathbb{E}\left[V_0(x^0-x^\star,y^0-y^\star)\right]$. Consequently,
\begin{align}\label{f 32}
\mathbb E\|z^k-z^\star\|_V^2
\le
2\,\mathbb E\|z^0-z^\star\|_V^2,
\qquad
\forall k\ge 0,
\end{align}
and
\begin{align}\label{f 33}
\mathbb E\!\left[\sum_{k=0}^{\infty}\|z^{k+1}-z^k\|_V^2\right]
\le
\frac{2}{1-2\gamma_1^2}\,
\mathbb E\|z^0-z^\star\|_V^2.
\end{align}
\end{lemma}
\begin{proof}
     Using \eqref{eq:Gamma-bound},\eqref{eq:Vk-descent}, we derive \eqref{f 30}.  By the definition of saddle point, we have $H(\hat{x}^{k+1},\hat{y}^{k+1},x^\star,y^\star) \ge 0$. Therefore,
     \begin{align}
         \begin{aligned}
             0 \leq& \mathbb{E}\left[V_k(x^k-x^\star,y^k-y^\star)\right]-\mathbb{E}\left[V_{k+1}(x^{k+1}-x^\star,y^{k+1}-y^\star)\right]\\
             &-\mathbb{E}\left[V(x^{k+1}-x^{k},y^k-y^{k-1})\right]. 
         \end{aligned}
          \label{f 34}
     \end{align}
   The remaining claims follow by combining \eqref{f 34} and \eqref{eq:V-coercive}, \eqref{eq:Vk-positive}.
\end{proof}

\subsection{A realized one-step bound}

To derive a rate for the actual stochastic iterates rather than the virtual full
updates, we need one additional decomposition. Define
\[
g_{P^{-1}}(\bar x):=\sum_{j=1}^m \frac{1}{p_j}g_j(\bar x_j),
\qquad
f^*_{Q^{-1}}(\bar y):=\sum_{i=1}^n \frac{1}{q_i}f_i^*(\bar y_i),
\]
and
\[
g_{P^{-1}-I}(\bar x):=
\sum_{j=1}^m \left(\frac{1}{p_j}-1\right)g_j(\bar x_j),
\qquad
f^*_{Q^{-1}-I}(\bar y):=
\sum_{i=1}^n \left(\frac{1}{q_i}-1\right)f_i^*(\bar y_i).
\]
For \(z=(x,y)\), define
\begin{align}
D_g^{P^{-1}-I}(\bar x;z)
&:=
g_{P^{-1}-I}(\bar x)-g_{P^{-1}-I}(x)
+\langle A(P^{-1}-I)(\bar x-x),y\rangle,
\label{eq:Dg-def}\\
D_{f^*}^{Q^{-1}-I}(\bar y;z)
&:=
f^*_{Q^{-1}-I}(\bar y)-f^*_{Q^{-1}-I}(y)
-\langle Ax,(Q^{-1}-I)(\bar y-y)\rangle .
\label{eq:Df-def}
\end{align}
We also define the stochastic residuals
\begin{align}\label{eq:uv-def}
u^k:=\hat x^{k+1}-x^k-P^{-1}(x^{k+1}-x^k),
\qquad
v^k:=\hat y^{k+1}-y^k-Q^{-1}(y^{k+1}-y^k).
\end{align}

\begin{lemma}\label{lem:realized-one-step}
Suppose Assumption~\ref{ass:basic} and \ref{ass:stepsizes-general} hold. For all \(k\ge 0\) and
\((x,y)\in\mathcal Z\), we have
\begin{align}
H(x^{k+1},y^{k+1},x,y)
\le\;&
D_g^{P^{-1}-I}(x^k;z)-D_g^{P^{-1}-I}(x^{k+1};z)
\nonumber\\
&+
D_{f^*}^{Q^{-1}-I}(y^k;z)-D_{f^*}^{Q^{-1}-I}(y^{k+1};z)
\nonumber\\
&+
\langle x,u^k\rangle_{\tau^{-1}}
+
\langle y,v^k\rangle_{\sigma^{-1}}
+\varepsilon_k
\nonumber\\
&+
V_k(x^k-x,y^k-y)
-
V_{k+1}(x^{k+1}-x,y^{k+1}-y)
\nonumber\\
&-
V(x^{k+1}-x^k,y^k-y^{k-1}).
\label{eq:realized-one-step}
\end{align}
Moreover, \(\mathbb E[\varepsilon_k]=0\).
\end{lemma}

\begin{proof}
We first compare the virtual and realized gap kernels. By definition,
\begin{align}
&H(\hat x^{k+1},\hat y^{k+1},x,y)
-
H(x^{k+1},y^{k+1},x,y)
\nonumber\\
 = \;&
g(\hat x^{k+1})-g(x^{k+1})
+\langle A(\hat x^{k+1}-x^{k+1}),y\rangle
\nonumber\\
&
+
f^*(\hat y^{k+1})-f^*(y^{k+1})
+\langle Ax,y^{k+1}-\hat y^{k+1}\rangle .
\label{eq:gap-difference}
\end{align}
Substituting \eqref{eq:gap-difference} into the virtual one-step inequality
\eqref{eq:virtual-one-step}, and rearranging the separable terms, gives
\begin{align}
    \begin{aligned}
        H(x^{k+1},y^{k+1},x,y) &\leq g(x^{k+1})-g(\hat{x}^{k+1})+\langle A(x^{k+1}-\hat{x}^{k+1}),y\rangle\\
        & +f^*(y^{k+1})-f^*(\hat{y}^{k+1})+\langle Ax, \hat{y}^{k+1}-y^{k+1}\rangle\\
        & +\langle A\hat{x}^{k+1},y\rangle -\langle Ax, \hat{y}^{k+1}\rangle\\
        & + \langle A(x-\hat{x}^{k+1}),\bar{y}^k\rangle + \langle A\bar{x}^{k+1},\hat{y}^{k+1}-y\rangle\\
        & + \frac{1}{2}\|x^k-x\|_{\tau^{-1}}^2-\frac{1}{2}\|\hat{x}^{k+1}-x^k\|_{\tau^{-1}}^2-\frac{1}{2}\|\hat{x}^{k+1}-x\|_{\tau^{-1}}^2\\
        &+ \frac{1}{2}\|y^k-y\|_{\sigma^{-1}}^2-\frac{1}{2}\|\hat{y}^{k+1}-y^k\|_{\sigma^{-1}}^2-\frac{1}{2}\|\hat{y}^{k+1}-y\|_{\sigma^{-1}}^2.
    \end{aligned}
    \label{eq 32}
\end{align}

For the first two rows in \eqref{eq 32}, we have the following equalities from direct computation.
\begin{align}
    \begin{aligned}
        & g(x^{k+1})-g(\hat{x}^{k+1})+\langle A(x^{k+1}-\hat{x}^{k+1}),y\rangle\\
        =&D_g^{P^{-1}-I}(x^k;z)-D_g^{P^{-1}-I}(x^{k+1};z)+R_g^k+ \langle x^k-\hat{x}^{k+1}-P^{-1}(x^k-x^{k+1}), A^\top y \rangle,\\
        & f^*(y^{k+1})-f^*(\hat{y}^{k+1})+\langle Ax, \hat{y}^{k+1}-y^{k+1}\rangle\\
        =&D_{f^*}^{Q^{-1}-I}(y^k;z)-D_{f^*}^{Q^{-1}-I}(y^{k+1};z)+R_{f^*}^k+ \langle Ax, \hat{y}^{k+1}-y^k-Q^{-1}(y^{k+1}-y^k)\rangle,
    \end{aligned}
    \label{eq 33}
\end{align}
where
\begin{align*}
R_g^k
&:=
g(x^k)-g(\hat x^{k+1})
+
g_{P^{-1}}(x^{k+1})-g_{P^{-1}}(x^k),\\
R_{f^*}^k
&:=
f^*(y^k)-f^*(\hat y^{k+1})
+
f^*_{Q^{-1}}(y^{k+1})-f^*_{Q^{-1}}(y^k).
\end{align*}

For the quadratic terms in \eqref{eq 32}, we have
\begin{align}
    \hat{\mathcal{Q}}_k=\mathcal{Q}_k+\hat{\mathcal{Q}}_k-\mathcal{Q}_k
    \label{Q_decompose}
\end{align}
where $ \hat{\mathcal{Q}}_k,\mathcal{Q}_k$ are defined in \eqref{Q_k_def}. Letting $\Delta \mathcal{Q}_k=\hat{\mathcal{Q}}_k-\mathcal{Q}_k$, and using the identity
\begin{align*}
    \|a-c\|_M^2-\|b-c\|_M^2=\|a\|_{M}^2-\|b\|_{M}^2+ 2\langle c, b-a\rangle_M,
\end{align*}
we obtain
\begin{align}
    \Delta \mathcal{Q}_k= \langle x, u^k\rangle_{\tau^{-1}}+\langle y, v^k\rangle_{\sigma^{-1}}+\delta_k,
    \label{eq Delta_Q}
\end{align}
where the zero-mean quadratic residual is
\begin{align}
\delta_k
:=\;&
\frac12
\left(
\|x^{k+1}-x^k\|_{\tau^{-1}P^{-1}}^2
-
\|\hat x^{k+1}-x^k\|_{\tau^{-1}}^2
\right)
\nonumber\\
&+
\frac12
\left(
\|y^{k+1}-y^k\|_{\sigma^{-1}Q^{-1}}^2
-
\|\hat y^{k+1}-y^k\|_{\sigma^{-1}}^2
\right)
\nonumber\\
&+
\frac12
\left(
\|x^{k+1}\|_{\tau^{-1}P^{-1}}^2
-
\|x^k\|_{\tau^{-1}P^{-1}}^2
+
\|x^k\|_{\tau^{-1}}^2
-
\|\hat x^{k+1}\|_{\tau^{-1}}^2
\right)
\nonumber\\
&+
\frac12
\left(
\|y^{k+1}\|_{\sigma^{-1}Q^{-1}}^2
-
\|y^k\|_{\sigma^{-1}Q^{-1}}^2
+
\|y^k\|_{\sigma^{-1}}^2
-
\|\hat y^{k+1}\|_{\sigma^{-1}}^2
\right).
\label{eq:delta-k-def}
\end{align}

Combining \eqref{eq 32} \eqref{eq 33} \eqref{Q_decompose} and \eqref{eq Delta_Q}, we have
\begin{align}
    \begin{aligned}
        H(x^{k+1},y^{k+1},x,y)
\le\;&
D_g^{P^{-1}-I}(x^k;z)-D_g^{P^{-1}-I}(x^{k+1};z)\\
&+
D_{f^*}^{Q^{-1}-I}(y^k;z)-D_{f^*}^{Q^{-1}-I}(y^{k+1};z)\\
&+
R_g^k+R_{f^*}^k
+
\mathcal B_k\\
&+
\mathcal Q_k+\langle x, u^k\rangle_{\tau^{-1}}+\langle y, v^k\rangle_{\sigma^{-1}}+\delta_k,
    \end{aligned}
    \label{eq:realized-pre}
\end{align}
where \(\mathcal B_k\) denotes the remaining bilinear part, which is
\begin{align}
    \begin{aligned}
        \mathcal{B}_k=&\langle A\hat{x}^{k+1},y\rangle -\langle Ax, \hat{y}^{k+1}\rangle + \langle A(x-\hat{x}^{k+1}),\bar{y}^k\rangle + \langle A\bar{x}^{k+1},\hat{y}^{k+1}-y\rangle\\
     &+\langle x^k-\hat{x}^{k+1}-P^{-1}(x^k-x^{k+1}), A^\top y \rangle+ \langle Ax, \hat{y}^{k+1}-y^k-Q^{-1}(y^{k+1}-y^k)\rangle
    \end{aligned}
\end{align}
Using
\[
\bar x^{k+1}=x^k+P^{-1}(x^{k+1}-x^k),
\qquad
\bar y^k=y^k+Q^{-1}(y^k-y^{k-1}),
\]
a direct rearrangement of the bilinear terms gives
\begin{align}
\mathcal B_k
=\;&
\langle A(\bar x^{k+1}-x),
Q^{-1}(y^{k+1}+y^{k-1}-2y^k)\rangle
\nonumber\\
&+
\langle A\bar x^{k+1},v^k\rangle
-
\langle Au^k,\bar y^k\rangle .
\label{eq:Bk-decomp}
\end{align}

Combining \eqref{eq:realized-pre} and \eqref{eq:Bk-decomp}, and using the definition
of \(\Gamma_k\), we obtain
\begin{align}
H(x^{k+1},y^{k+1},x,y)
\le\;&
D_g^{P^{-1}-I}(x^k;z)-D_g^{P^{-1}-I}(x^{k+1};z)
\nonumber\\
&+
D_{f^*}^{Q^{-1}-I}(y^k;z)-D_{f^*}^{Q^{-1}-I}(y^{k+1};z)
\nonumber\\
&+
\langle x,u^k\rangle_{\tau^{-1}}
+
\langle y,v^k\rangle_{\sigma^{-1}}
+
\varepsilon_k
+
\Gamma_k,
\label{eq:realized-with-Gamma}
\end{align}
where
\begin{align}\label{eq:epsilon-def}
\varepsilon_k
:=
R_g^k+R_{f^*}^k+\delta_k
+\langle A\bar x^{k+1},v^k\rangle
-\langle Au^k,\bar y^k\rangle .
\end{align}

It remains to verify that \(\varepsilon_k\) has zero expectation. This follows from
the block update rules. Indeed, for the primal part of $\varepsilon_k$, we have
\[
\mathbb E[u^k\mid\mathcal F_k]=0,
\qquad
\mathbb E[R_g^k\mid\mathcal F_k]=0,
\qquad
\mathbb E[\delta_k^x\mid\mathcal F_k]=0.
\]
For the dual part of $\varepsilon_k$, we have
\[
\mathbb E[v^k\mid\mathcal F_{k+\frac{1}{2}}]=0,
\qquad
\mathbb E[R_{f^*}^k\mid\mathcal F_{k+\frac{1}{2}}]=0,
\qquad
\mathbb E[\delta_k^y\mid\mathcal F_{k+\frac{1}{2}}]=0.
\]
Here \(\delta_k^x\) and \(\delta_k^y\) are the \(x\)- and \(y\)-parts of
\(\delta_k\), respectively. Since \(\bar y^k\) is \(\mathcal F_k\)-measurable and
\(\bar x^{k+1}\) is \(\mathcal F_{k+\frac{1}{2}}\)-measurable, we also have
\[
\mathbb E\langle Au^k,\bar y^k\rangle=0,
\qquad
\mathbb E\langle A\bar x^{k+1},v^k\rangle=0.
\]
Therefore \(\mathbb E[\varepsilon_k]=0\).

Finally, combining \eqref{eq:realized-with-Gamma} with
\eqref{eq:Vk-descent} gives \eqref{eq:realized-one-step}.
\end{proof}

To control the accumulated correction terms in Lemma~\ref{lem:realized-one-step},
we need to estimate $D_g^{P^{-1}-I}$ and $D_{f^*}^{Q^{-1}-I}$.
These terms do not admit a direct descent structure and therefore cannot be
handled by the Lyapunov argument alone.

The idea is to relate them to the primal--dual gap and to distances to the
solution set using saddle-point optimality and convexity.
To decouple these two effects, we allow two (possibly different) saddle points,
leading to the following estimate.

\begin{lemma}\label{lem:correction-bound}
Suppose Assumption~\ref{ass:basic} holds. Let
\(z_1^\star=(x_1^\star,y_1^\star)\) and
\(z_2^\star=(x_2^\star,y_2^\star)\) be two saddle points. Define
\[
\underline p:=\min_{j\in[m]}p_j,
\qquad
\underline q:=\min_{i\in[n]}q_i,
\]
and
\begin{align*}
    \begin{aligned}
        \gamma_4^2
:=
\left\|
\tau^{1/2}(P^{-1/2}-P^{1/2})A^\top Q^{1/2}\sigma^{1/2}
\right\|\\
\gamma_5^2
:=
\left\|
\tau^{1/2}P^{1/2}A^\top(Q^{-1/2}-Q^{1/2})\sigma^{1/2}
\right\|
    \end{aligned}
\end{align*}
Then, for any positive integer \(K\) and any \(z=(x,y)\in\mathcal Z\),
\begin{align}
D_g^{P^{-1}-I}(x^0;z)-D_g^{P^{-1}-I}(x^K;z)
\le\;&
\left(\frac{1}{\underline p}-1\right)
\left(
\Phi(x^0,y_1^\star)-\Phi(x_1^\star,y_1^\star)
\right)
\nonumber\\
&\hspace{-8 em}+
\frac{\gamma_4^2}{2}
\left(
2\|x^K-x_2^\star\|_{\tau^{-1}P^{-1}}^2
+
2\|x^0-x_2^\star\|_{\tau^{-1}P^{-1}}^2
+
\|y-y_1^\star\|_{\sigma^{-1}Q^{-1}}^2
\right),
\label{eq:Dg-correction-bound}
\end{align}
and
\begin{align}
D_{f^*}^{Q^{-1}-I}(y^0;z)-D_{f^*}^{Q^{-1}-I}(y^K;z)
\le\;&
\left(\frac{1}{\underline q}-1\right)
\left(
\Phi(x_1^\star,y_1^\star)-\Phi(x_1^\star,y^0)
\right)
\nonumber\\
&\hspace{-8 em}+
\frac{\gamma_5^2}{2}
\left(
\|x-x_1^\star\|_{\tau^{-1}P^{-1}}^2
+
2\|y^K-y_2^\star\|_{\sigma^{-1}Q^{-1}}^2
+
2\|y^0-y_2^\star\|_{\sigma^{-1}Q^{-1}}^2
\right).
\label{eq:Df-correction-bound}
\end{align}
Consequently,
\begin{align}
&D_g^{P^{-1}-I}(x^0;z)-D_g^{P^{-1}-I}(x^K;z)
+
D_{f^*}^{Q^{-1}-I}(y^0;z)-D_{f^*}^{Q^{-1}-I}(y^K;z)
\nonumber\\
&\qquad\le
\alpha H(x^0,y^0,x_1^\star,y_1^\star)
+
\beta\|z^K-z_2^\star\|_V^2
+
\beta\|z^0-z_2^\star\|_V^2
+
\frac{\beta}{2}\|z-z_1^\star\|_V^2,
\label{eq:correction-bound}
\end{align}
where
\[
\alpha:=
\max\left\{
\frac{1}{\underline p}-1,\,
\frac{1}{\underline q}-1
\right\},
\qquad
\beta:=\max\{\gamma_4^2,\gamma_5^2\}.
\]
\end{lemma}

\begin{proof}
We first prove the bound for the primal correction term. By definition,
\begin{align}
    \begin{aligned}
        D_g^{P^{-1}-I}(x^0;z)-D_g^{P^{-1}-I}(x^K;z)
=&
g_{P^{-1}-I}(x^0)-g_{P^{-1}-I}(x^K)\\
&+
\langle A(P^{-1}-I)(x^0-x^K),y\rangle.
    \end{aligned}
    \label{eq:Dg-diff}
\end{align}
Since \(z_1^\star\) is a saddle point, we have
\[
-(A^j)^\top y_1^\star\in \partial g_j((x_1^\star)_j),
\qquad j\in[m].
\]
Hence, by convexity of \(g_j\),
\[
g_j(x_j^K)
\ge
g_j((x_1^\star)_j)
-
\left\langle (A^j)^\top y_1^\star,
x_j^K-(x_1^\star)_j
\right\rangle .
\]
Multiplying by \(p_j^{-1}-1\) and summing over \(j\) gives
\[
g_{P^{-1}-I}(x^K)
\ge
g_{P^{-1}-I}(x_1^\star)
-
\langle A(P^{-1}-I)(x^K-x_1^\star),y_1^\star\rangle .
\]
Substituting this inequality into \eqref{eq:Dg-diff}, we obtain
\begin{align}
&D_g^{P^{-1}-I}(x^0;z)-D_g^{P^{-1}-I}(x^K;z)
\nonumber\\
&\le
g_{P^{-1}-I}(x^0)-g_{P^{-1}-I}(x_1^\star)
+
\langle A(P^{-1}-I)(x^0-x_1^\star),y_1^\star\rangle
\nonumber\\
&\qquad
+
\langle A(P^{-1}-I)(x^0-x^K),y-y_1^\star\rangle .
\label{eq:Dg-split}
\end{align}
For the first three terms on the right-hand side, using again the saddle-point
subgradient condition gives
\begin{align}
&g_{P^{-1}-I}(x^0)-g_{P^{-1}-I}(x_1^\star)
+
\langle A(P^{-1}-I)(x^0-x_1^\star),y_1^\star\rangle
\nonumber\\
&=
\sum_{j=1}^m
\left(\frac{1}{p_j}-1\right)
\left[
g_j(x_j^0)-g_j((x_1^\star)_j)
+
\langle A^j(x_j^0-(x_1^\star)_j),y_1^\star\rangle
\right]
\nonumber\\
&\le
\left(\frac{1}{\underline p}-1\right)
\left[
g(x^0)-g(x_1^\star)
+
\langle A(x^0-x_1^\star),y_1^\star\rangle
\right]
\nonumber\\
&=
\left(\frac{1}{\underline p}-1\right)
\left(
\Phi(x^0,y_1^\star)-\Phi(x_1^\star,y_1^\star)
\right).
\label{eq:Dg-first-bound}
\end{align}
For the remaining cross term, by the definition of \(\gamma_4^2\),
\begin{align}
\left|
\langle A(P^{-1}-I)(x^0-x^K),y-y_1^\star\rangle
\right|
&\le
\frac{\gamma_4^2}{2}
\left(
\|x^0-x^K\|_{\tau^{-1}P^{-1}}^2
+
\|y-y_1^\star\|_{\sigma^{-1}Q^{-1}}^2
\right)
\nonumber\\
&\hspace{-7 em}\le
\frac{\gamma_4^2}{2}
\left(
2\|x^K-x_2^\star\|_{\tau^{-1}P^{-1}}^2
+
2\|x^0-x_2^\star\|_{\tau^{-1}P^{-1}}^2
+
\|y-y_1^\star\|_{\sigma^{-1}Q^{-1}}^2
\right).
\label{eq:Dg-cross-bound}
\end{align}
Combining \eqref{eq:Dg-split}, \eqref{eq:Dg-first-bound}, and
\eqref{eq:Dg-cross-bound} proves \eqref{eq:Dg-correction-bound}.

The proof of \eqref{eq:Df-correction-bound} is analogous. Indeed, since
\(z_1^\star\) is a saddle point,
\[
A_i x_1^\star\in \partial f_i^*((y_1^\star)_i),
\qquad i\in[n],
\]
and repeating the same argument with \(f_i^*\), \(Q\), and the operator norm
constant \(\gamma_5^2\) yields \eqref{eq:Df-correction-bound}.

It remains to combine the two estimates. By the definition of the gap kernel,
\[
H(x^0,y^0,x_1^\star,y_1^\star)
=
\Phi(x^0,y_1^\star)-\Phi(x_1^\star,y^0),
\]
and therefore
\begin{align*}
H(x^0,y^0,x_1^\star,y_1^\star)
=
\bigl[
\Phi(x^0,y_1^\star)-\Phi(x_1^\star,y_1^\star)
\bigr]
+
\bigl[
\Phi(x_1^\star,y_1^\star)-\Phi(x_1^\star,y^0)
\bigr].
\end{align*}
 Summing \eqref{eq:Dg-correction-bound} and \eqref{eq:Df-correction-bound}, and using the
definitions of \(\alpha\) and \(\beta\), yields \eqref{eq:correction-bound}.
\end{proof}

To control the random linear terms in Lemma~\ref{lem:realized-one-step}, we introduce
an auxiliary sequence driven by the stochastic errors \(u^k\) and \(v^k\). This
allows us to bound their accumulated contribution uniformly over a bounded test
region \(B\).

\begin{lemma}\label{lem:random-linear-bound}
Suppose Assumptions~\ref{ass:basic} and~\ref{ass:stepsizes-general} hold.
Let \(z^\star\in\mathcal Z^\star\) be a saddle point. Given
\(\tilde z^0=(\tilde x^0,\tilde y^0)\), define
\[
\tilde x^{k+1}=\tilde x^k+Pu^k,
\qquad
\tilde y^{k+1}=\tilde y^k+Qv^k,
\qquad k\ge 0.
\]
Then, for any \(K\ge 1\), any \(z=(x,y)\in\mathcal Z\), and any set
\(B\subseteq\mathcal Z\), we have
\begin{align}
\sum_{k=0}^{K-1}\langle x,u^k\rangle_{\tau^{-1}}
&\le
\frac12\|\tilde x^0-x\|_{\tau^{-1}P^{-1}}^2
+\frac12\sum_{k=0}^{K-1}\|u^k\|_{\tau^{-1}P}^2
+\sum_{k=0}^{K-1}\langle \tilde x^k,u^k\rangle_{\tau^{-1}},
\label{eq:uk-linear-bound}
\\
\sum_{k=0}^{K-1}\langle y,v^k\rangle_{\sigma^{-1}}
&\le
\frac12\|\tilde y^0-y\|_{\sigma^{-1}Q^{-1}}^2
+\frac12\sum_{k=0}^{K-1}\|v^k\|_{\sigma^{-1}Q}^2
+\sum_{k=0}^{K-1}\langle \tilde y^k,v^k\rangle_{\sigma^{-1}} .
\label{eq:vk-linear-bound}
\end{align}
Moreover,
\begin{align}\label{eq:uv-summability}
\mathbb E\!\left[
\sum_{k=0}^{\infty}
\left(
\|u^k\|_{\tau^{-1}P}^2
+
\|v^k\|_{\sigma^{-1}Q}^2
\right)
\right]
\le
\frac{2}{1-2\gamma_1^2}
\mathbb E\|z^0-z^\star\|_V^2 .
\end{align}
Consequently,
\begin{align}
    \begin{aligned}
        \mathbb{E}&\left[\sup_{z\in B}\left\{ \sum_{k=0}^{K-1}\langle x, u^k\rangle_{\tau^{-1}}+\langle y, v^k\rangle_{\sigma^{-1}}\right\}\right]\\
        &\hspace{1em}\leq \mathbb{E}\sup_{z \in B}\left\{\frac{1}{2}\|\tilde{x}^0-x \|_{\tau^{-1}P^{-1}}^2+\frac{1}{2}\|\tilde{y}^0-y\|_{\sigma^{-1}Q^{-1}}^2\right\} 
        + \frac{1}{1-2\gamma_1^2}\mathbb{E}\|z^0-z^\star\|_V^2.
    \end{aligned}
    \label{eq:random-linear-sup-bound}
\end{align}
\end{lemma}

\begin{proof}
We first prove \eqref{eq:uk-linear-bound}. By the definition of
\(\tilde x^{k+1}\),
\[
\frac12\|\tilde x^{k+1}-x\|_{\tau^{-1}P^{-1}}^2
=
\frac12\|\tilde x^k-x\|_{\tau^{-1}P^{-1}}^2
+
\langle u^k,\tilde x^k-x\rangle_{\tau^{-1}}
+
\frac12\|u^k\|_{\tau^{-1}P}^2 .
\]
Rearranging and summing over \(k=0,\ldots,K-1\) gives
\[
\sum_{k=0}^{K-1}\langle x-\tilde x^k,u^k\rangle_{\tau^{-1}}
\le
\frac12\|\tilde x^0-x\|_{\tau^{-1}P^{-1}}^2
+
\frac12\sum_{k=0}^{K-1}\|u^k\|_{\tau^{-1}P}^2 ,
\]
which implies \eqref{eq:uk-linear-bound}. The proof of
\eqref{eq:vk-linear-bound} is identical.

Next, by the definition of \(u^k\),
\[
u^k=\hat x^{k+1}-x^k-P^{-1}(x^{k+1}-x^k).
\]
we have
\begin{align}
    \begin{aligned}
     \|u^k\|_{\tau^{-1}P}^2=&\|x^k-\hat{x}^{k+1}\|_{\tau^{-1}P}^2+\|x^{k+1}-x^{k}\|_{\tau^{-1}P^{-1}}^2\\
     &-2\langle x^k-\hat{x}^{k+1},x^k-x^{k+1}\rangle_{\tau^{-1}}.
    \end{aligned}
\end{align}
Since $x^k,\hat{x}^{k+1}$ are $\mathcal{F}_k$ measurable and $\mathbb E[x^k-x^{k+1}\mid \mathcal{F}_k]=P(x^k-\hat{x}^{k+1})$, we have
\begin{align}
    \begin{aligned}
            \mathbb{E} \left[\langle x^k-\hat{x}^{k+1},x^k-x^{k+1}\rangle_{\tau^{-1}}\mid \mathcal{F}_k\right]
            &= \langle x^k-\hat{x}^{k+1}, P(x^k-\hat{x}^{k+1})\rangle_{\tau^{-1}}\\
            &= \|x^k-\hat{x}^{k+1}\|_{\tau^{-1}P}^2.
    \end{aligned}
\end{align}
Therefore,
\begin{align}
    \begin{aligned}
          \mathbb{E}\left[\|u^k\|_{\tau^{-1}P}^2\right]&= \mathbb{E}\left[\|x^{k+1}-x^{k}\|_{\tau^{-1}P^{-1}}^2\right]-\mathbb{E}\left[\|x^k-\hat{x}^{k+1}\|_{\tau^{-1}P}^2\right]\\
          &\leq \mathbb{E}\left[\|x^{k+1}-x^{k}\|_{\tau^{-1}P^{-1}}^2\right].
    \end{aligned}
     \label{f 68}
\end{align}
Similarly,
\begin{align}
    \begin{aligned}
            \mathbb{E}\left[\|v^k\|_{\sigma^{-1}Q}^2\right]&= \mathbb{E}\left[\|y^{k+1}-y^{k}\|_{\sigma^{-1}Q^{-1}}^2\right]-\mathbb{E}\left[\|y^k-\hat{y}^{k+1}\|_{\sigma^{-1}Q}^2\right]\\
            &\leq \mathbb{E}\left[\|y^{k+1}-y^{k}\|_{\sigma^{-1}Q^{-1}}^2\right].
    \end{aligned}
    \label{f 69}
\end{align}
Combining \eqref{f 68} and \eqref{f 69}, we have
\begin{align}
\mathbb{E}\left[\sum_{k=0}^\infty\left(\|u^k\|_{\tau^{-1}P}^2+\|v^k\|_{\sigma^{-1}Q}^2\right)\right] \leq \mathbb{E}\left[\sum_{k=0}^\infty\|z^{k+1}-z^k\|_V^2\right] \leq \frac{2}{1-2\gamma_1^2}\mathbb{E}\|z^0-z^\star\|_V^2,
\label{f 70}
\end{align}
where the last inequality comes from \eqref{f 33}.

Finally, summing \eqref{eq:uk-linear-bound} and \eqref{eq:vk-linear-bound} and taking
the supremum over \(z\in B\), we obtain
\begin{align}
    \begin{aligned}
        &\sup_{z\in B}\left\{
\sum_{k=0}^{K-1}
\left(
\langle x,u^k\rangle_{\tau^{-1}}
+
\langle y,v^k\rangle_{\sigma^{-1}}
\right)\right\}
\\
&\le
\sup_{z\in B}
\left\{
\frac12\|\tilde x^0-x\|_{\tau^{-1}P^{-1}}^2
+
\frac12\|\tilde y^0-y\|_{\sigma^{-1}Q^{-1}}^2
\right\}
\\
&\quad
+\frac12
\sum_{k=0}^{K-1}
\left(
\|u^k\|_{\tau^{-1}P}^2
+
\|v^k\|_{\sigma^{-1}Q}^2
\right)
+
M_K ,
    \end{aligned}
    \label{eq correlation_terms}
\end{align}
where
\[
M_K:=
\sum_{k=0}^{K-1}
\left(
\langle \tilde x^k,u^k\rangle_{\tau^{-1}}
+
\langle \tilde y^k,v^k\rangle_{\sigma^{-1}}
\right).
\]
Since \(\tilde x^k\) and \(\tilde y^k\) are measurable with respect to $\mathcal{F}_k$ and \(u^k,v^k\) are conditionally mean-zero with respect to $\mathcal{F}_k$ and $\mathcal{F}_{k+\frac{1}{2}}$ respectively, \(\mathbb E[M_K]=0\). Taking expectation on \eqref{eq correlation_terms} and using
\eqref{eq:uv-summability} gives \eqref{eq:random-linear-sup-bound}.
\end{proof}

\subsection{Ergodic $\mathcal O(1/K)$ rate}

Define the ergodic averages
\begin{align}\label{eq:ergodic-averages}
X^K:=\frac1K\sum_{k=1}^K x^k,
\qquad
Y^K:=\frac1K\sum_{k=1}^K y^k .
\end{align}
For a bounded set $B\subseteq\mathcal Z$, define the restricted primal--dual gap
\begin{align}\label{eq:restricted-gap}
G_B(\bar x,\bar y)
:=
\sup_{z=(x,y)\in B} H(\bar x,\bar y,x,y).
\end{align}

The next theorem is the main result of this section.

\begin{theorem}\label{thm:general-convex-gap}
Suppose Assumptions~\ref{ass:basic} and \ref{ass:stepsizes-general} hold.
Let $B\subseteq\mathcal Z$ be bounded. Then there exists a constant $C_B>0$,
depending only on the initial point, the saddle-point set, the step sizes, and $B$,
such that
\begin{align}\label{eq:main-gap-rate}
\mathbb E\!\left[G_B(X^K,Y^K)\right]
\le
\frac{C_B}{K},
\qquad
\forall K\ge 1 .
\end{align}
Hence the ergodic sequence generated by DSPDHG converges at rate
$\mathcal O(1/K)$ in the expected restricted primal--dual gap.
\end{theorem}

\begin{proof}
Fix $z=(x,y)\in B$. Since $H(\cdot,\cdot,x,y)$ is convex in its first two
arguments, Jensen's inequality gives
\[
H(X^K,Y^K,x,y)
\le
\frac1K\sum_{k=0}^{K-1}
H(x^{k+1},y^{k+1},x,y).
\]
Therefore,
\begin{align}\label{eq:ergodic-gap-jensen}
K\,G_B(X^K,Y^K)
\le
\sup_{z\in B}
\sum_{k=0}^{K-1}
H(x^{k+1},y^{k+1},x,y).
\end{align}

Summing the realized one-step inequality \eqref{eq:realized-one-step} from
$k=0$ to $K-1$, using \eqref{eq:ergodic-gap-jensen}, and dropping the nonnegative
terms
\[
V_K(x^K-x,y^K-y),
\qquad
\sum_{k=0}^{K-1}V(x^{k+1}-x^k,y^k-y^{k-1}),
\]
we obtain
\begin{align}
K\,\mathbb E\!\left[G_B(X^K,Y^K)\right]
\le\;&
\mathbb E\!\left[
\sup_{z\in B}
\left\{
D_g^{P^{-1}-I}(x^0;z)-D_g^{P^{-1}-I}(x^K;z)
\right.\right.
\nonumber\\
&\hspace{7em}
\left.\left.
+
D_{f^*}^{Q^{-1}-I}(y^0;z)-D_{f^*}^{Q^{-1}-I}(y^K;z)
\right\}
\right]
\nonumber\\
&+
\mathbb E\!\left[
\sup_{z\in B}
\left\{\sum_{k=0}^{K-1}
\left(
\langle x,u^k\rangle_{\tau^{-1}}
+
\langle y,v^k\rangle_{\sigma^{-1}}
\right)\right\}
\right]
\nonumber\\
&+
\mathbb E\!\left[
\sup_{z\in B}
V_0(x^0-x,y^0-y)
\right].
\label{eq:ergodic-proof-start}
\end{align}
Here the martingale residuals $\varepsilon_k$ vanish after taking expectation.

Since $y^{-1}=y^0$, the definition of $V_0$ gives
\[
V_0(x^0-x,y^0-y)
=
\frac12\|z-z^0\|_V^2 .
\]
Next, applying the weighted correction bound \eqref{eq:correction-bound} with
$z_1^\star=z_2^\star=z^\star\in\mathcal Z^\star$ yields
\begin{equation}
    \begin{aligned}
        &\mathbb E\!\left[
\sup_{z\in B}
\left\{
D_g^{P^{-1}-I}(x^0;z)-D_g^{P^{-1}-I}(x^K;z)
+
D_{f^*}^{Q^{-1}-I}(y^0;z)-D_{f^*}^{Q^{-1}-I}(y^K;z)
\right\}
\right]\\
&\le
\alpha\,\mathbb E\!\left[H(x^0,y^0,x^\star,y^\star)\right]
+
\beta\,\mathbb E\|z^K-z^\star\|_V^2
+
\beta\,\mathbb E\|z^0-z^\star\|_V^2
+
\frac{\beta}{2}
\mathbb E\!\left[\sup_{z\in B}\|z-z^\star\|_V^2\right].
    \end{aligned}
    \label{eq:ergodic-correction-bound}
\end{equation}
Moreover, by  \eqref{eq:random-linear-sup-bound} with
$\tilde x^0=x^0$ and $\tilde y^0=y^0$,
\begin{align}
&\mathbb E\!\left[
\sup_{z\in B}
\left\{\sum_{k=0}^{K-1}
\left(
\langle x,u^k\rangle_{\tau^{-1}}
+
\langle y,v^k\rangle_{\sigma^{-1}}
\right)\right\}
\right]
\nonumber\\
&\le
\frac12\mathbb E\!\left[
\sup_{z\in B}
\|z-z^0\|_V^2
\right]
+
\frac{1}{1-2\gamma_1^2}
\mathbb E\|z^0-z^\star\|_V^2 .
\label{eq:ergodic-stochastic-bound}
\end{align}

Combining \eqref{eq:ergodic-proof-start}, \eqref{eq:ergodic-correction-bound},
and \eqref{eq:ergodic-stochastic-bound}, and using
\[
\mathbb E\|z^K-z^\star\|_V^2
\le
2\,\mathbb E\|z^0-z^\star\|_V^2,
\]
we obtain
\begin{align}
K\,\mathbb E\!\left[G_B(X^K,Y^K)\right]
\le\;&
\alpha\,\mathbb E\!\left[H(x^0,y^0,x^\star,y^\star)\right]
+
\left(
3\beta+\frac{1}{1-2\gamma_1^2}
\right)
\mathbb E\|z^0-z^\star\|_V^2
\nonumber\\
&+
\frac{\beta}{2}
\mathbb E\!\left[
\sup_{z\in B}\|z-z^\star\|_V^2
\right]
+
\mathbb E\!\left[
\sup_{z\in B}\|z-z^0\|_V^2
\right].
\label{eq:ergodic-final-bound}
\end{align}
The right-hand side is finite because $B$ is bounded. Defining this quantity as
$C_B$ and dividing both sides by $K$ proves \eqref{eq:main-gap-rate}.
\end{proof}

\begin{remark}\label{rem:CB-explicit}
An explicit admissible constant in Theorem~\ref{thm:general-convex-gap} is
\begin{align}
C_B
:={}&
\alpha\,\mathbb E\!\left[H(x^0,y^0,x^\star,y^\star)\right]
+
\left(
3\beta+\frac{1}{1-2\gamma_1^2}
\right)
\mathbb E\|z^0-z^\star\|_V^2
\nonumber\\
&+
\frac{\beta}{2}\,
\mathbb E\!\left[\sup_{z\in B}\|z-z^\star\|_V^2\right]
+
\mathbb E\!\left[\sup_{z\in B}\|z-z^0\|_V^2\right],
\label{eq:CB-explicit}
\end{align}
where $z^\star\in\mathcal Z^\star$ is arbitrary.
\end{remark}

\begin{remark}
Theorem~\ref{thm:general-convex-gap} extends the classical
$\mathcal O(1/K)$ expected restricted-gap guarantee for SPDHG to the doubly
stochastic setting. The main additional difficulty is that stochasticity appears
on both primal and dual sides, which requires controlling two correction sequences,
$u^k$ and $v^k$, instead of a single one-sided correction term.
\end{remark}

\section{Linear Convergence under Smoothed Quadratic Growth}
\label{sec:linear}

We now show that a restarted variant of DSPDHG achieves  linear convergence 
under a regularity condition formulated in terms of a smoothed primal--dual gap.
Unlike strong convexity--concavity assumptions imposed on specific components of the
objective, the condition below is expressed directly in terms of the saddle-point
geometry, and is therefore more flexible.

\subsection{Smoothed quadratic growth}

Recall the smoothed gap functional
\[
G_\mu(\bar z,\dot z)
:=
\sup_{z\in\mathcal Z}
\left\{
H(\bar x,\bar y,x,y)
-
\frac{\mu}{2}\|z-\dot z\|_V^2
\right\},
\qquad
\bar z=(\bar x,\bar y),\ \dot z\in\mathcal Z,
\]
introduced in \eqref{eq:smoothed-gap}. The notion of smoothed quadratic growth with
respect to $G_\mu$ was introduced in \cite{fercoq2022quadratic} and has since been
used in several related works.

\begin{definition}[Smoothed quadratic growth]
Suppose $\mu>0$. We say that the smoothed duality gap $G_\mu$ satisfies quadratic
growth on a  set $\mathcal R$ with parameter $\zeta$ if, for every
$\bar z\in\mathcal R$ and every $z^\star\in\mathcal Z^\star$,
\begin{align}\label{eq:sqg-main}
G_\mu(\bar z,z^\star)
\ge
\zeta\,\dist_V(\bar z,\mathcal Z^\star)^2.
\end{align}
\end{definition}

This condition lower bounds the smoothed gap by the squared distance to the solution
set in the weighted geometry induced by $\|\cdot\|_V$. It plays the role of an error
bound condition and serves as the key ingredient for establishing linear convergence.

\begin{remark}
The smoothed quadratic growth condition holds in a variety of settings, including
strongly convex--strongly concave problems \cite{fercoq2022quadratic}, linear programs
\cite{fercoq2022quadratic}, quadratic programs (QPs) \cite{lu2025practical}, and piecewise linear quadratic (PLQ) functions
\cite{fercoq2022quadratic}. In particular, the two classes of convex problems
considered in our numerical experiments, namely soft-margin SVM and model predictive
control (MPC), fall into the PLQ and QP categories, respectively, and therefore satisfy
this condition automatically.
\end{remark}

Other regularity conditions have also been considered in the literature, such as metric
subregularity (see \cite[Assumption 2]{alacaoglu2022convergence}) and
the two-sided quadratic functional growth (QFG) condition (see
\cite[Definition 2.7]{Hquadraticgrowth}). For quadratic programming,
\cite{lu2025practical} showed that smoothed quadratic growth can be implied by metric
subregularity. However, whether such an implication holds more generally remains open.
Moreover, to the best of our knowledge, there is currently no result clarifying whether
smoothed quadratic growth and two-sided QFG imply each other.

\subsection{Restarted scheme}

We consider a restarted outer scheme built from the ergodic output of DSPDHG.
Given an outer iterate $\mathbf{Z}^t\in\mathcal Z$, we run DSPDHG for $K$ inner
iterations initialized at $z^0=\mathbf{Z}^t$, and define
\begin{align}\label{eq:restart-output}
\mathbf{Z}^{t+1}
=
\frac{1}{K}\sum_{k=1}^K z^k
=
\left(
\frac{1}{K}\sum_{k=1}^K x^k,\,
\frac{1}{K}\sum_{k=1}^K y^k
\right).
\end{align}
Equivalently, if $\operatorname{DSPDHG}(\mathbf{Z}^t,K)$ denotes the one-epoch
ergodic output obtained by running $K$ iterations of Algorithm~\ref{alg:DSPDHG}
from the initial point $\mathbf{Z}^t$, then
\[
\mathbf{Z}^{t+1}=\operatorname{DSPDHG}(\mathbf{Z}^t,K).
\]

\begin{algorithm}[t]
\caption{Restarted DSPDHG}
\label{alg:restarted-DSPDHG}
\begin{algorithmic}[1]
\STATE \textbf{Input:} initial point $\mathbf{Z}^0\in\mathcal Z$ and epoch length $K$
\FOR{$t=0,1,2,\dots$}
    \STATE Run Algorithm~\ref{alg:DSPDHG} for $K$ iterations with $z^0=\mathbf{Z}^t$
    \STATE Set $\mathbf{Z}^{t+1}:=\frac1K\sum_{k=1}^K z^k$
\ENDFOR
\end{algorithmic}
\end{algorithm}

\begin{assumption}\label{ass:sqg-main}
There exist constants $\mu>0$ and $\zeta>0$, together with a compact set
$\mathcal{R}\subseteq\mathcal{Z}$, such that $G_\mu$ satisfies quadratic growth
on $\mathcal{R}$ with parameter $\zeta$. Moreover, the outer iterates satisfy
$\mathbf{Z}^t\in\mathcal{R}$ for all $t\ge 0$.
\end{assumption}

\begin{remark}
The requirement $\mathbf{Z}^t\in\mathcal{R}$ for all $t\ge 0$ may appear somewhat
restrictive, but it is needed for the linear convergence analysis. In the strongly convex
case, where one may take $\mathcal{R}=\mathcal{Z}$, this condition is automatic.
In other settings, however, such as quadratic programming \cite{lu2025practical},
the set $\mathcal{R}$ must be chosen as a proper compact subset of $\mathcal{Z}$,
and verifying that the iterates remain in $\mathcal{R}$ can be nontrivial. The difficulty
comes from the fact that the trajectory is only controlled in expectation, so one cannot
exclude low-probability excursions far away from $\mathcal{R}$. A similar assumption
is also imposed in \cite[Assumption 2]{alacaoglu2022convergence} in the linear
convergence analysis of SPDHG.
\end{remark}

\subsection{Linear convergence of the restarted method}
We can now state the main theorem of this section.
\begin{theorem}\label{thm:linear-restarted}
Suppose Assumptions~\ref{ass:basic}, \ref{ass:stepsizes-general}, and
\ref{ass:sqg-main} hold. Let $\{\mathbf Z^t\}_{t\ge 0}$ be generated by
Algorithm~\ref{alg:restarted-DSPDHG}. Assume that the epoch length $K$ satisfies
\begin{align}\label{eq:K-condition-full}
K
\ge
\max\left\{
\frac{4+2\beta}{\mu},
\ \frac{100}{\zeta}\left(\alpha\zeta+5\beta+2+\frac{1}{1-2\gamma_1^2}\right),
\ \frac{400\beta}{\zeta}
\right\}.
\end{align}
Then the restarted DSPDHG scheme converges linearly. More precisely, there exists
a constant $M>0$ such that
\begin{align}\label{eq:linear-gap}
\mathbb E\!\left[
G_\mu(\mathbf Z^t,\Pi^{t-1})
\right]
\le
M\,5^{-t},
\qquad \forall t\ge 1,
\end{align}
where $\Pi^{t-1}:=\proj_{\mathcal Z^\star}(\mathbf Z^{t-1})$. Consequently,
\begin{align}\label{eq:linear-distance}
\mathbb E\!\left[\dist_V(\mathbf Z^t,\mathcal Z^\star)^2\right]
\le
\frac{M}{\zeta}\,5^{-t},
\qquad \forall t\ge 1.
\end{align}
\end{theorem}
\begin{remark}
Theorem~\ref{thm:linear-restarted} shows that DSPDHG achieves a geometric rate once
the smoothed quadratic growth condition holds. The restart mechanism converts the
ergodic $\mathcal O(1/K)$ behavior of the inner method into a linear decrease of the
smoothed gap and, consequently, of the distance to the saddle-point set.

The constants $\alpha$ and $\beta$ come from the one-epoch estimate and the correction
bounds. They are not needed to run the inner DSPDHG iterations; they only enter the
theoretical lower bound on the restart length $K$.
\end{remark}

The rest of this section presents the proof of Theorem~\ref{thm:linear-restarted}.

\subsubsection{A one-epoch estimate}

We next derive an estimate for one restart epoch. Throughout this subsection, we write
\[
Z^K:=(X^K,Y^K):=\frac1K\sum_{k=1}^K z^k,
\qquad
z^k=(x^k,y^k),
\]
where $Z^K$ denotes the ergodic output of one epoch, while $z^K$ denotes the last
inner iterate.

Summing the realized one-step inequality \eqref{eq:realized-one-step} from
$k=0$ to $K-1$, and using the convexity of
$H(\cdot,\cdot,x,y)$ in its first two arguments, we have that for any
$z=(x,y)\in\mathcal Z$,
\begin{align}\label{eq:epoch-start}
\begin{aligned}
& K H(X^K,Y^K,x,y) \\
\le\;&
D_g^{P^{-1}-I}(x^0;z)-D_g^{P^{-1}-I}(x^K;z) +
D_{f^*}^{Q^{-1}-I}(y^0;z)-D_{f^*}^{Q^{-1}-I}(y^K;z)
\\
&+
\sum_{k=0}^{K-1}
\left(
\langle x,u^k\rangle_{\tau^{-1}}
+
\langle y,v^k\rangle_{\sigma^{-1}}
+
\varepsilon_k
\right)+
V_0(x^0-x,y^0-y).
\end{aligned}
\end{align}
This inequality is the starting point for the linear convergence analysis.

The following lemma gives a one-epoch estimate for the smoothed gap.

\begin{lemma}\label{lem:epoch-smoothed-bound}
Suppose Assumptions~\ref{ass:basic} and \ref{ass:stepsizes-general} hold. Then, for any
$z_1^\star,z_2^\star\in\mathcal Z^\star$ and any $K$ satisfying
$\mu K\geq 2\beta+4$,
\begin{align}
\mathbb E\!\left[G_\mu(Z^K,z_2^\star)\right]
\le\;&
\frac{\alpha}{K}\,
\mathbb E\!\left[G_\mu(z^0,z_1^\star)\right]
+
\frac{\beta}{K}\,
\mathbb E\!\left[\|z_1^\star-z_2^\star\|_V^2\right]
\nonumber\\
&+
\frac{1}{K}
\left(
3\beta
+
2
+
\frac{1}{1-2\gamma_1^2}
\right)
\mathbb E\!\left[\|z^0-z_2^\star\|_V^2\right].
\label{eq:epoch-smoothed-bound}
\end{align}
\end{lemma}

\begin{proof}
Starting from \eqref{eq:epoch-start}, we use \eqref{eq:correction-bound} with two arbitrary saddle points $z_1^\star,z_2^\star$  and \eqref{eq:uk-linear-bound}, \eqref{eq:vk-linear-bound} with the auxiliary initialization
$\tilde x^0=x^0$ and $\tilde y^0=y^0$. This gives, for every
$z=(x,y)\in\mathcal Z$,
\begin{align}
 & K H(X^K,Y^K,x,y)  \nonumber\\
\le\;&
\alpha H(x^0,y^0,x_1^\star,y_1^\star)
+
\beta\|z^K-z_2^\star\|_V^2
+
\beta\|z^0-z_2^\star\|_V^2
+
\frac{\beta}{2}\|z-z_1^\star\|_V^2
+
\|z-z^0\|_V^2
\nonumber\\
&+
\frac12\sum_{k=0}^{K-1}
\left(
\|u^k\|_{\tau^{-1}P}^2
+
\|v^k\|_{\sigma^{-1}Q}^2
\right) +
\sum_{k=0}^{K-1}
\left(
\langle \tilde x^k,u^k\rangle_{\tau^{-1}}
+
\langle \tilde y^k,v^k\rangle_{\sigma^{-1}}
+
\varepsilon_k
\right).
\label{eq:epoch-intermediate}
\end{align}
The expectation of the last summation is zero by $\mathbb E[M_K]=0$ and $\mathbb E[\varepsilon_k=0]$. Dividing \eqref{eq:epoch-intermediate} by $K$ and subtracting
$\frac{\mu}{2}\|z-z_2^\star\|_V^2$ from both sides yields
\begin{align}
&H(X^K,Y^K,x,y)
-
\frac{\mu}{2}\|z-z_2^\star\|_V^2
\nonumber\\
\le\;&
\frac{\alpha}{K}H(x^0,y^0,x_1^\star,y_1^\star)
+
\frac{\beta}{K}\|z^K-z_2^\star\|_V^2
+
\frac{\beta}{K}\|z^0-z_2^\star\|_V^2
+
\frac{\beta}{2K}\|z-z_1^\star\|_V^2 \nonumber\\
&
+
\frac{1}{K}\|z-z^0\|_V^2
-
\frac{\mu}{2}\|z-z_2^\star\|_V^2
+
\frac{1}{2K}\sum_{k=0}^{K-1}
\left(
\|u^k\|_{\tau^{-1}P}^2
+
\|v^k\|_{\sigma^{-1}Q}^2
\right)
\nonumber\\
&+
\frac{1}{K}
\sum_{k=0}^{K-1}
\left(
\langle \tilde x^k,u^k\rangle_{\tau^{-1}}
+
\langle \tilde y^k,v^k\rangle_{\sigma^{-1}}
+
\varepsilon_k
\right).
\label{eq:epoch-smoothed-pre-sup}
\end{align}
Since $G_\mu(z^0,z_1^\star)=\sup_{z}\{H(x^0,y^0,x,y)-\frac{\mu}{2}\|z-z_1^\star\|_V^2\}$,
\[
H(x^0,y^0,x_1^\star,y_1^\star)
\le
G_\mu(z^0,z_1^\star).
\]

It remains to bound the quadratic terms involving $z$. Using
\[
\|z-z_1^\star\|_V^2
\le
2\|z-z_2^\star\|_V^2
+
2\|z_1^\star-z_2^\star\|_V^2,
\]
we have
\begin{align}
&\sup_{z\in\mathcal Z}
\left\{
\frac{\beta}{2K}\|z-z_1^\star\|_V^2
+
\frac{1}{K}\|z-z^0\|_V^2
-
\frac{\mu}{2}\|z-z_2^\star\|_V^2
\right\}
\nonumber\\
&\le
\frac{\beta}{K}\|z_1^\star-z_2^\star\|_V^2
+
\sup_{z\in\mathcal Z}
\left\{
-\left(\frac{\mu}{2}-\frac{\beta}{K}\right)\|z-z_2^\star\|_V^2
+
\frac{1}{K}\|z-z^0\|_V^2
\right\}.
\label{eq:quad-sup-1}
\end{align}
Since $\mu K>2\beta+2$, the quadratic function in the last supremum is concave.
The standard quadratic maximization formula gives
\begin{align}\label{eq:quad-sup-2}
& \sup_{z\in\mathcal Z}
\left\{
-\left(\frac{\mu}{2}-\frac{\beta}{K}\right)\|z-z_2^\star\|_V^2
+
\frac{1}{K}\|z-z^0\|_V^2
\right\}  \nonumber \\
\le \; &
\frac{\mu K-2\beta}{K(\mu K-2\beta-2)}
\|z^0-z_2^\star\|_V^2  \leq \frac{2}{K} \|z^0-z_2^\star\|_V^2 .
\end{align}
where the last inequality holds because of the assumption $\mu K\geq 2\beta+4$. Combining \eqref{eq:quad-sup-1} and \eqref{eq:quad-sup-2}, taking the
supremum over $z\in\mathcal Z$ in \eqref{eq:epoch-smoothed-pre-sup}, and  then using the fact $H(z^0,z_1^\star) \leq G_\mu(z^0,z_1^\star)$, we obtain
\begin{align}
G_\mu(Z^K,z_2^\star)
\le\;&
\frac{\alpha}{K}G_\mu(z^0,z_1^\star)
+
\frac{\beta}{K}\|z^K-z_2^\star\|_V^2
+
\frac{\beta}{K}\|z^0-z_2^\star\|_V^2
\nonumber\\
&+
\frac{\beta}{K}\|z_1^\star-z_2^\star\|_V^2
+
\frac{2}{K}
\|z^0-z_2^\star\|_V^2
+
\frac{1}{2K}\sum_{k=0}^{K-1}
\left(
\|u^k\|_{\tau^{-1}P}^2
+
\|v^k\|_{\sigma^{-1}Q}^2
\right)
\nonumber\\
&+
\frac{1}{K}
\sum_{k=0}^{K-1}
\left(
\langle \tilde x^k,u^k\rangle_{\tau^{-1}}
+
\langle \tilde y^k,v^k\rangle_{\sigma^{-1}}
+
\varepsilon_k
\right).
\label{eq:epoch-before-expectation}
\end{align}

Taking expectations in \eqref{eq:epoch-before-expectation}, the last summation
vanishes. Moreover, by the uniform second-moment bound \eqref{f 32},
\[
\mathbb E\|z^K-z_2^\star\|_V^2
\le
2\,\mathbb E\|z^0-z_2^\star\|_V^2,
\]
and by the square-summability estimation \eqref{eq:uv-summability},
\[
\mathbb E\!\left[
\sum_{k=0}^{K-1}
\left(
\|u^k\|_{\tau^{-1}P}^2
+
\|v^k\|_{\sigma^{-1}Q}^2
\right)
\right]
\le
\frac{2}{1-2\gamma_1^2}
\mathbb E\|z^0-z_2^\star\|_V^2 .
\]
Therefore,
\begin{align}
\mathbb E\!\left[G_\mu(Z^K,z_2^\star)\right]
\le\;&
\frac{\alpha}{K}
\mathbb E\!\left[G_\mu(z^0,z_1^\star)\right]
+
\frac{\beta}{K}
\mathbb E\!\left[\|z_1^\star-z_2^\star\|_V^2\right]
\nonumber\\
&+
\frac{1}{K}
\left(
3\beta
+
2
+
\frac{1}{1-2\gamma_1^2}
\right)
\mathbb E\!\left[\|z^0-z_2^\star\|_V^2\right],
\end{align}
which proves \eqref{eq:epoch-smoothed-bound}.
\end{proof}

\subsubsection{Inter-epoch recursion}

We now derive a recursion across consecutive restart epochs. 
\begin{lemma}
    Suppose Assumptions~\ref{ass:basic}, \ref{ass:stepsizes-general}, and
\ref{ass:sqg-main} hold. For each outer iterate,
define $\Pi^t:=\proj_{\mathcal Z^\star}(\mathbf Z^t).$
Then, for all $K$ satisfying
$\mu K\geq 2\beta+4$ and $t\ge 2$,
\begin{align}
\mathbb E\!\left[G_\mu(\mathbf Z^{t+1},\Pi^t)\right]
\le\;&
\frac{4\beta}{K\zeta}\,
\mathbb E\!\left[
G_\mu(\mathbf Z^{t-1},\Pi^{t-2})
\right]
\nonumber\\
&+
\frac{1}{K\zeta}
\left(
\alpha\zeta
+
5\beta
+
2
+
\frac{1}{1-2\gamma_1^2}
\right)
\mathbb E\!\left[
G_\mu(\mathbf Z^t,\Pi^{t-1})
\right].
\label{eq:second-order-recursion}
\end{align}
\end{lemma}
\begin{proof}
    Applying Lemma~\ref{lem:epoch-smoothed-bound} to the epoch initialized at
$z^0=\mathbf Z^t$, with ergodic output $Z^K=\mathbf Z^{t+1}$, and choosing $z_2^\star=\Pi^t$ and $z_1^\star=\Pi^{t-1},$
we obtain, for all $t\ge 1$,
\begin{align}
&\mathbb E\!\left[G_\mu(\mathbf Z^{t+1},\Pi^t)\right] \nonumber\\
\le\;&
\frac{\alpha}{K}\,
\mathbb E\!\left[G_\mu(\mathbf Z^t,\Pi^{t-1})\right]
+
\frac{\beta}{K}\,
\mathbb E\!\left[\|\Pi^t-\Pi^{t-1}\|_V^2\right]
+
\frac{c_0(K)}{K}\,
\mathbb E\!\left[\dist_V(\mathbf Z^t,\mathcal Z^\star)^2\right],
\label{eq:pre-recursion}
\end{align}
where
\[
c_0(K)
:=
3\beta
+
2
+
\frac{1}{1-2\gamma_1^2}.
\]

It remains to control the distance between two consecutive projected saddle points.
Consider the previous epoch, initialized at $\mathbf Z^{t-1}$ and producing the
ergodic output $\mathbf Z^t$. By the triangle inequality,
\[
\|\Pi^t-\Pi^{t-1}\|_V^2
\le
2\|\mathbf Z^t-\Pi^t\|_V^2
+
2\|\mathbf Z^t-\Pi^{t-1}\|_V^2.
\]
Taking expectations gives
\begin{align}
\mathbb E\!\left[\|\Pi^t-\Pi^{t-1}\|_V^2\right]
\le\;&
2\,\mathbb E\!\left[\dist_V(\mathbf Z^t,\mathcal Z^\star)^2\right]
+
2\,\mathbb E\!\left[\|\mathbf Z^t-\Pi^{t-1}\|_V^2\right].
\label{eq:projection-diff-1}
\end{align}
Since $\mathbf Z^t$ is the ergodic average of the inner iterates of the previous
epoch, Jensen's inequality yields
\[
\|\mathbf Z^t-\Pi^{t-1}\|_V^2
\le
\frac1K\sum_{k=1}^K
\|z^k-\Pi^{t-1}\|_V^2,
\]
where $z^k$ denotes the inner iterates generated during the epoch initialized at
$\mathbf Z^{t-1}$. 
Applying the uniform bound in \eqref{f 32} to the previous
epoch with $z^\star=\Pi^{t-1}$ gives
\[
\mathbb E\!\left[\|z^k-\Pi^{t-1}\|_V^2\right]
\le
2\,\mathbb E\!\left[\|\mathbf Z^{t-1}-\Pi^{t-1}\|_V^2\right]
=
2\,\mathbb E\!\left[\dist_V(\mathbf Z^{t-1},\mathcal Z^\star)^2\right].
\]
Therefore,
\begin{align}
    \mathbb E\|\mathbf Z^t-\Pi^{t-1}\|_V^2 \leq 2\,\mathbb E\!\left[\dist_V(\mathbf Z^{t-1},\mathcal Z^\star)^2\right]
    \label{eq: average_bound_expectation}
\end{align}
Substituting this bound into \eqref{eq:projection-diff-1}, we obtain
\begin{align}
\mathbb E\!\left[\|\Pi^t-\Pi^{t-1}\|_V^2\right]
\le\;&
2\,\mathbb E\!\left[\dist_V(\mathbf Z^t,\mathcal Z^\star)^2\right]
+
4\,\mathbb E\!\left[\dist_V(\mathbf Z^{t-1},\mathcal Z^\star)^2\right].
\label{eq:projection-diff}
\end{align}
Combining \eqref{eq:pre-recursion} and \eqref{eq:projection-diff}, we get
\begin{align}
\mathbb E\!\left[G_\mu(\mathbf Z^{t+1},\Pi^t)\right]  
\le\;&
\frac{\alpha}{K}\,
\mathbb E\!\left[G_\mu(\mathbf Z^t,\Pi^{t-1})\right] \nonumber\\
&+
\frac{2\beta+c_0(K)}{K}\,
\mathbb E\!\left[\dist_V(\mathbf Z^t,\mathcal Z^\star)^2\right]
+
\frac{4\beta}{K}\,
\mathbb E\!\left[\dist_V(\mathbf Z^{t-1},\mathcal Z^\star)^2\right].
\label{eq:pre-second-order-recursion}
\end{align}
By Assumption~\ref{ass:sqg-main}, for any saddle point $z^\star\in\mathcal Z^\star$,
\[
\dist_V(\bar z,\mathcal Z^\star)^2
\le
\frac{1}{\zeta}G_\mu(\bar z,z^\star).
\]
Applying this inequality with $(\bar z,z^\star)=(\mathbf Z^t,\Pi^{t-1})$ and
$(\bar z,z^\star)=(\mathbf Z^{t-1},\Pi^{t-2})$, we obtain, for all $t\ge 2$,
\begin{align}
\mathbb E\!\left[G_\mu(\mathbf Z^{t+1},\Pi^t)\right]
\le\;&
\frac{4\beta}{K\zeta}\,
\mathbb E\!\left[
G_\mu(\mathbf Z^{t-1},\Pi^{t-2})
\right]
\nonumber\\
&+
\frac{1}{K\zeta}
\left(
\alpha\zeta
+
2\beta
+
c_0(K)
\right)
\mathbb E\!\left[
G_\mu(\mathbf Z^t,\Pi^{t-1})
\right].
\label{eq:second-order-recursion-pre}
\end{align}
Equivalently, using the definition of $c_0(K)$,
\begin{align}
\mathbb E\!\left[G_\mu(\mathbf Z^{t+1},\Pi^t)\right]
\le\;&
\frac{4\beta}{K\zeta}\,
\mathbb E\!\left[
G_\mu(\mathbf Z^{t-1},\Pi^{t-2})
\right]
\nonumber\\
&+
\frac{1}{K\zeta}
\left(
\alpha\zeta
+
5\beta
+
2
+
\frac{1}{1-2\gamma_1^2}
\right)
\mathbb E\!\left[
G_\mu(\mathbf Z^t,\Pi^{t-1})
\right],
\end{align}
for all $t\ge 2$.
\end{proof}

Now we are ready to prove theorem~\ref{thm:linear-restarted}.
\begin{proof}[Proof of Theorem~\ref{thm:linear-restarted}]
Define
\[
a_t
:=
\mathbb E\!\left[
G_\mu(\mathbf Z^t,\Pi^{t-1})
\right],\quad t\ge 1,\quad a_0:=  E\!\left[
G_\mu(\mathbf Z^0,\Pi^{0})
\right]
.
\]
By \eqref{eq:second-order-recursion}, the sequence $\{a_t\}_{t\ge 1}$ satisfies a
second-order linear recursion. To simplify the constants, define
\[
C
:=
\frac{1}{\zeta}
\max\left\{
\alpha\zeta+5\beta+2+\frac{1}{1-2\gamma_1^2},
\ 4\beta
\right\}.
\]
The definition of $C$, together withe the condition \eqref{eq:K-condition-full}, implies $K \geq 100C$,
and thus both coefficients in \eqref{eq:second-order-recursion} are bounded by
$1/100$. Hence
\begin{align}\label{eq:a-recursion}
a_{t+1}
\le
\frac{1}{100}(a_t+a_{t-1}),
\qquad t\ge 2.
\end{align}
Next, let $r:=1/5$, then $r^2=\frac{1}{25}
>
\frac{1+r}{100}.$ By choosing $M:=\max\{a_0,5a_1,25a_2\}$, we will show by induction that 
\[
a_t\le M\,r^{-t},
\qquad \forall t\ge 1.
\]
The claim holds for
$t=1,2$ by the definition of $M$. Suppose it holds for $t-1$ and $t$. Then
\[
a_{t+1}
\le
\frac{1}{100}(a_t+a_{t-1})
\le
\frac{M}{100}(r^t+r^{t-1})
=
Mr^{t-1}\frac{1+r}{100}
\le
Mr^{t+1},
\]
where the last inequality follows from $r^2> (1+r)/100$. This completes the
induction.

Since $\Pi^{t-1}\in\mathcal Z^\star$, Assumption~\ref{ass:sqg-main} gives
\[
\zeta\,\dist_V(\mathbf Z^t,\mathcal Z^\star)^2
\le
G_\mu(\mathbf Z^t,\Pi^{t-1}).
\]
Taking expectations and using \eqref{eq:linear-gap} yields
\eqref{eq:linear-distance}.
\end{proof}

Finally, as a consequence of the linear convergence result, we can show that the sequence \(\{\mathbf{Z}^t\}_{t\ge 0}\) remains bounded in expectation within a ball centered at \(\mathbf{Z}^0\). This provides further support for the assumption \(\mathbf{Z}^t\in\mathcal{R}\) for all \(t\ge 0\) in Assumption~\ref{ass:sqg-main}, and also yields an estimate of the radius of \(\mathcal{R}\).

\begin{proposition}\label{prop}
   Under the same conditions as  theorem~\ref{thm:linear-restarted}, we have
   \begin{align}
       \mathbb E\left\|\mathbf{Z}^N-\mathbf{Z}^0\right\|_V^2 \leq \frac{20M}{\zeta}, \quad \forall N \ge 0.
   \end{align}
\end{proposition}
\begin{proof}
According to Cauchy-Schwarz inequality, we have
\begin{align*}
    \|\mathbf{Z}^N-\mathbf{Z}^0\|_V^2 \leq \left( \sum_{t=1}^N \|\mathbf{Z}^t-\mathbf{Z}^{t-1}\|_V \right)^2 & \leq \left( \sum_{t=1}^N 2^{-t} \right) \left( \sum_{t=1}^N 2^t\|\mathbf{Z}^t-\mathbf{Z}^{t-1}\|_V^2 \right) \\
    & \leq \sum_{t=1}^N 2^t\|\mathbf{Z}^t-\mathbf{Z}^{t-1}\|_V^2.
\end{align*}
Besides, for any $t \ge 1$, 
    \begin{align}
        \begin{aligned}
            \mathbb E\|\mathbf{Z}^t-\mathbf{Z}^{t-1}\|_V^2 
            &\leq 2\mathbb E\|\mathbf{Z}^t-\Pi^{t-1}\|_V^2 +2\mathbb E\left[\dist_V (\mathbf{Z}^{t-1},\mathcal{Z}^\star)^2\right]\\
            &\leq 6 \mathbb E\left[\dist_V (\mathbf{Z}^{t-1},\mathcal{Z}^\star)^2\right] \leq \frac{6M}{\zeta}5^{-(t-1)},
        \end{aligned}
    \end{align}
where the second inequality comes from \eqref{eq: average_bound_expectation} and the third inequality comes from \eqref{eq:linear-distance}. Hence, we have
\begin{align*}
    \mathbb E \|\mathbf{Z}^N-\mathbf{Z}^0\|_V^2 \leq \sum_{t=1}^N 2^t \mathbb E\|\mathbf{Z}^t-\mathbf{Z}^{t-1}\|_V^2 \leq \sum_{t=1}^N \frac{12M}{\zeta} \left(\frac{2}{5}\right)^{t-1} \leq \frac{20M}{\zeta}.
\end{align*}
\end{proof}

\section{Numerical Experiments}\label{sec:numerics}

In this section, we illustrate the practical performance of DSPDHG and its restarted
variant RDSPDHG on two classes of convex problems: soft-margin SVM under entrywise sampling and model predictive control (MPC) under blockwise sampling. Our experiments serve two
purposes. First, they verify the sublinear convergence behavior predicted by the theory
for general convex problems. Second, they demonstrate the effectiveness of restarting
on problem instances that exhibit smoothed duality-gap quadratic growth.

Throughout this section, we compare deterministic PDHG, stochastic PDHG, doubly
stochastic PDHG, and their restarted variants. We use uniform sampling on both primal
and dual blocks. 
All methods are initialized from the same starting point. The step sizes are chosen according to the uniform rule
\[
\sigma = 0.99\,\frac{q}{\Lambda_{r,s}},
\qquad
\tau = 0.99\,\frac{p}{\Lambda_{r,s}},
\]
where $\Lambda_{r,s}$ is defined in \eqref{eq:Lambda_rs} with $r = nq$ and $s=pm$. For SPDHG, we take $p=1$,
and for PDHG we take $p=q=1$, in which case the rule reduces to the standard PDHG
choice $\sigma=\tau=\frac{0.99}{\|A\|}.$ 
For the restarted methods, we adopt an adaptive restart rule based on the relative KKT
residual, following \cite{lu2025practical, huang2024restarted}: a restart is triggered whenever
\begin{align}
    \mathrm{relKKT}(z^k)\le 0.8\,\mathrm{relKKT}(z^0).
    \label{eq:relKKT}
\end{align}
This criterion is simple to implement and proved robust in our experiments.
The relative error measures the gap between the objective value at the current iterate
and the optimal objective value computed by the commercial solver COPT
\cite{ge2022cardinal}.

\subsection{Soft-margin SVM}
\begin{figure}[t]
    \centering
    \begin{minipage}[b]{0.32\textwidth}
        \centering
        \includegraphics[width=\textwidth]{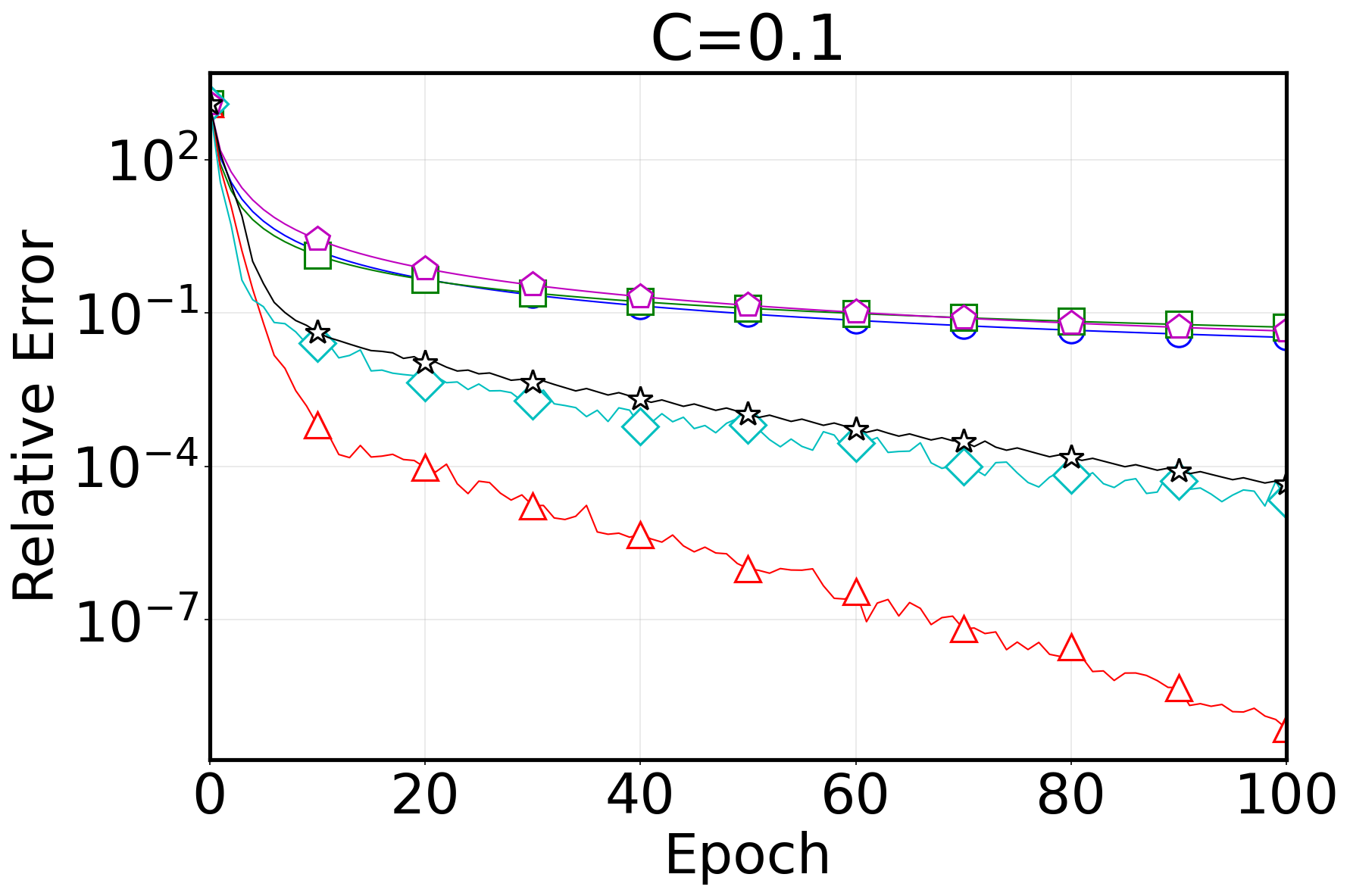}
    \end{minipage}\hfill
    \begin{minipage}[b]{0.32\textwidth}
        \centering
        \includegraphics[width=\textwidth]{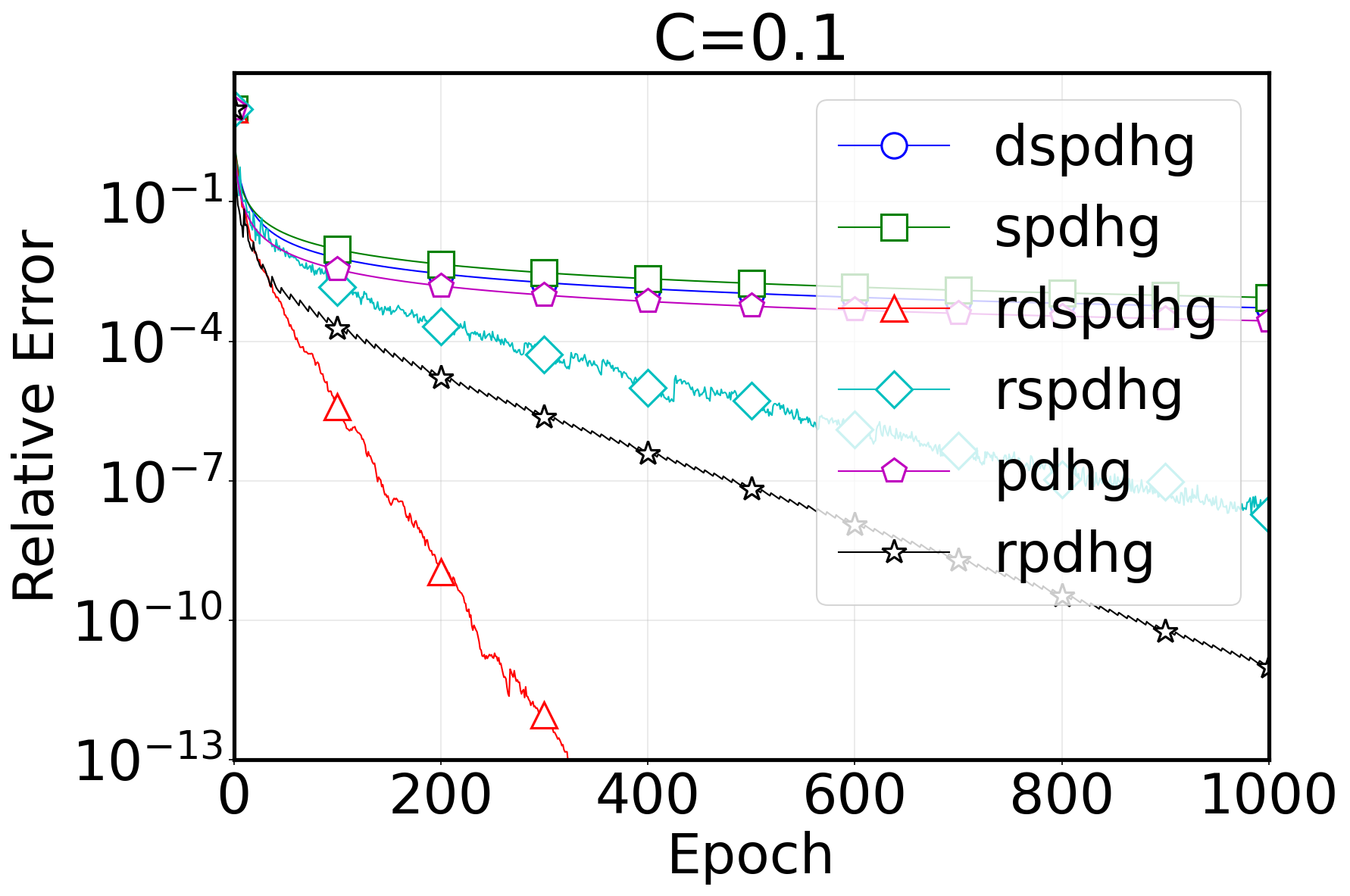}
    \end{minipage}\hfill
    \begin{minipage}[b]{0.32\textwidth}
        \centering
        \includegraphics[width=\textwidth]{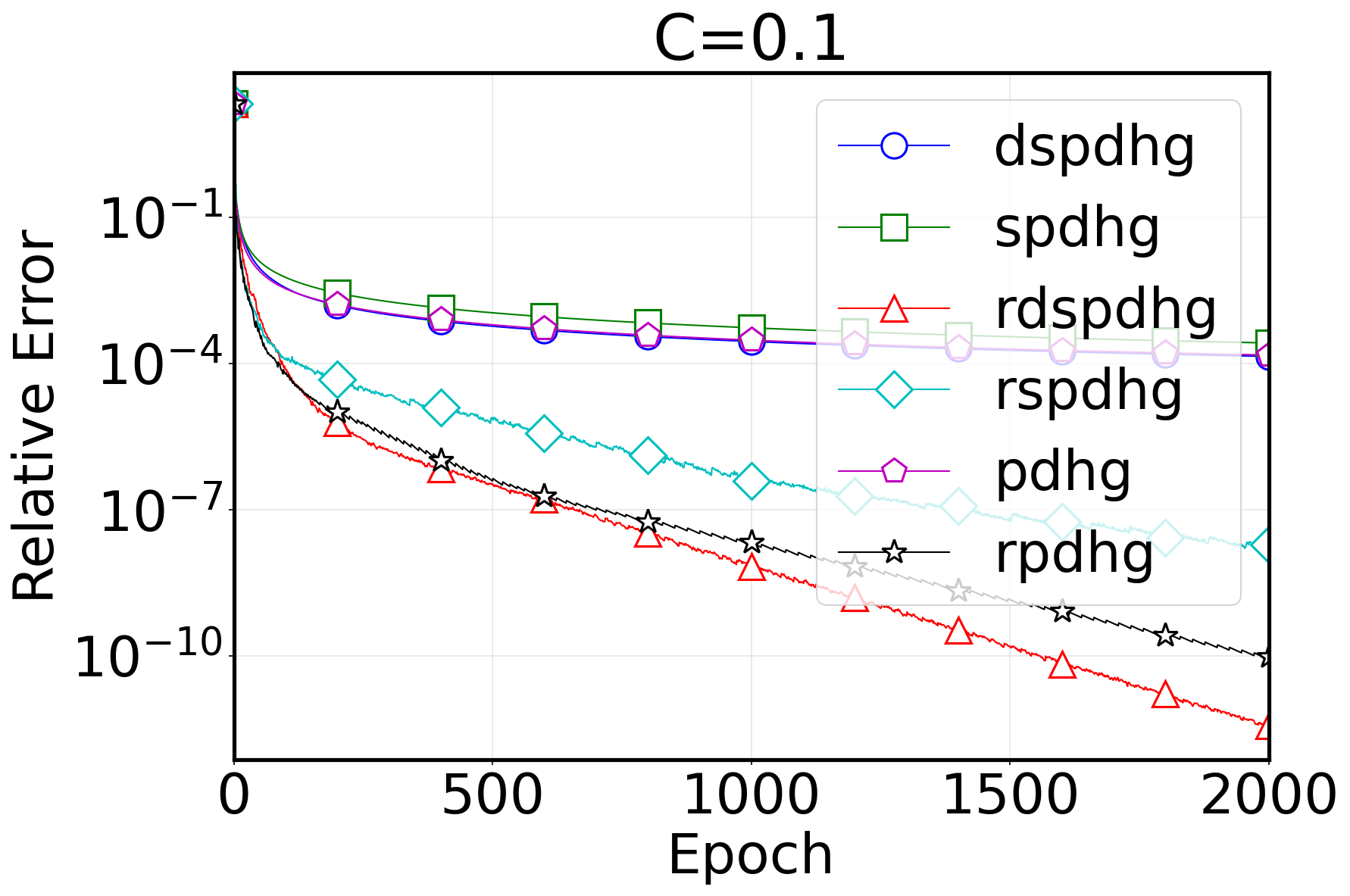}
    \end{minipage}

    \vspace{1em} 

    \begin{minipage}[b]{0.32\textwidth}
        \centering
        \includegraphics[width=\textwidth]{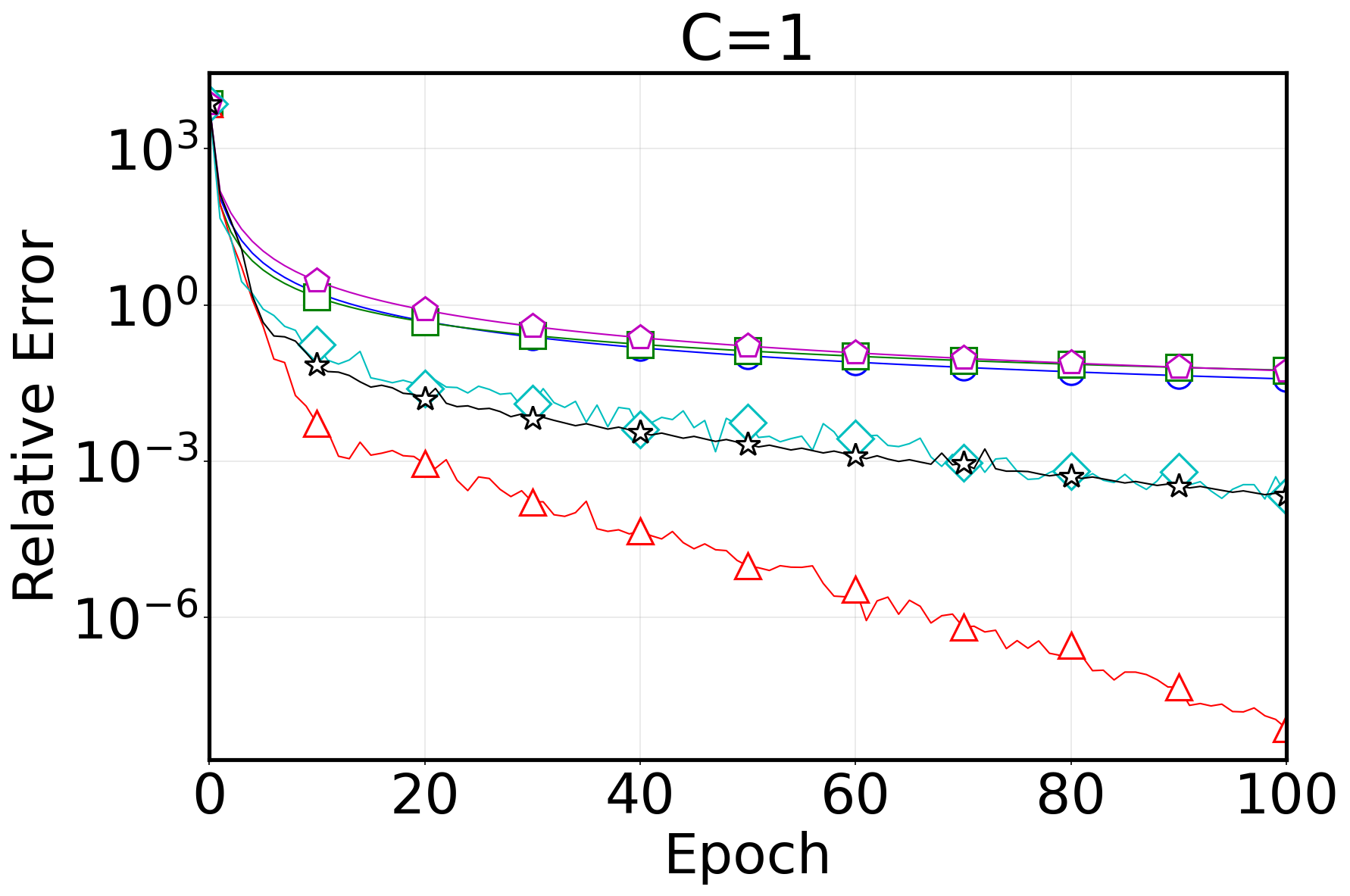}
    \end{minipage}\hfill
    \begin{minipage}[b]{0.32\textwidth}
        \centering
        \includegraphics[width=\textwidth]{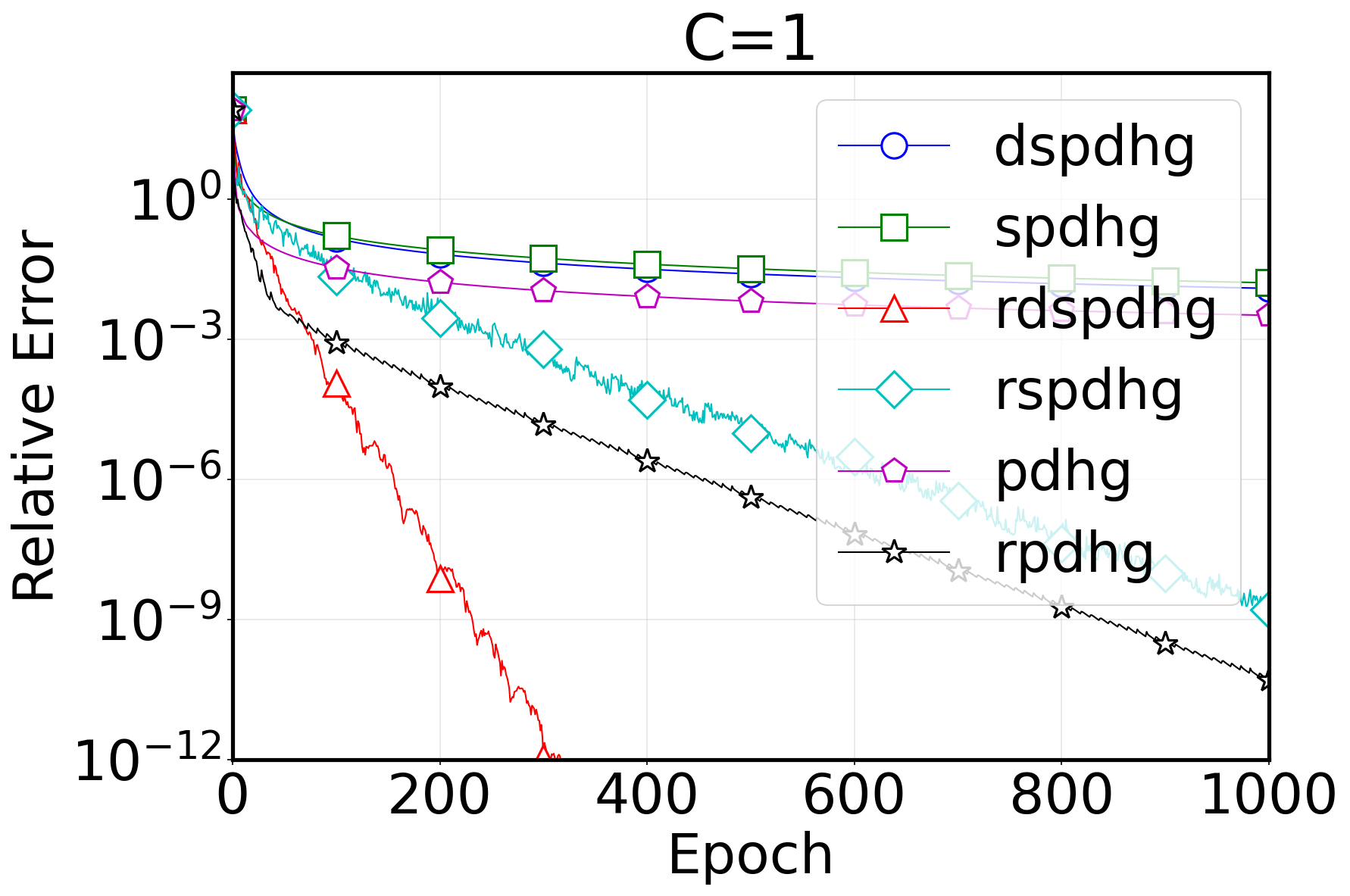}
    \end{minipage}\hfill
    \begin{minipage}[b]{0.32\textwidth}
        \centering
        \includegraphics[width=\textwidth]{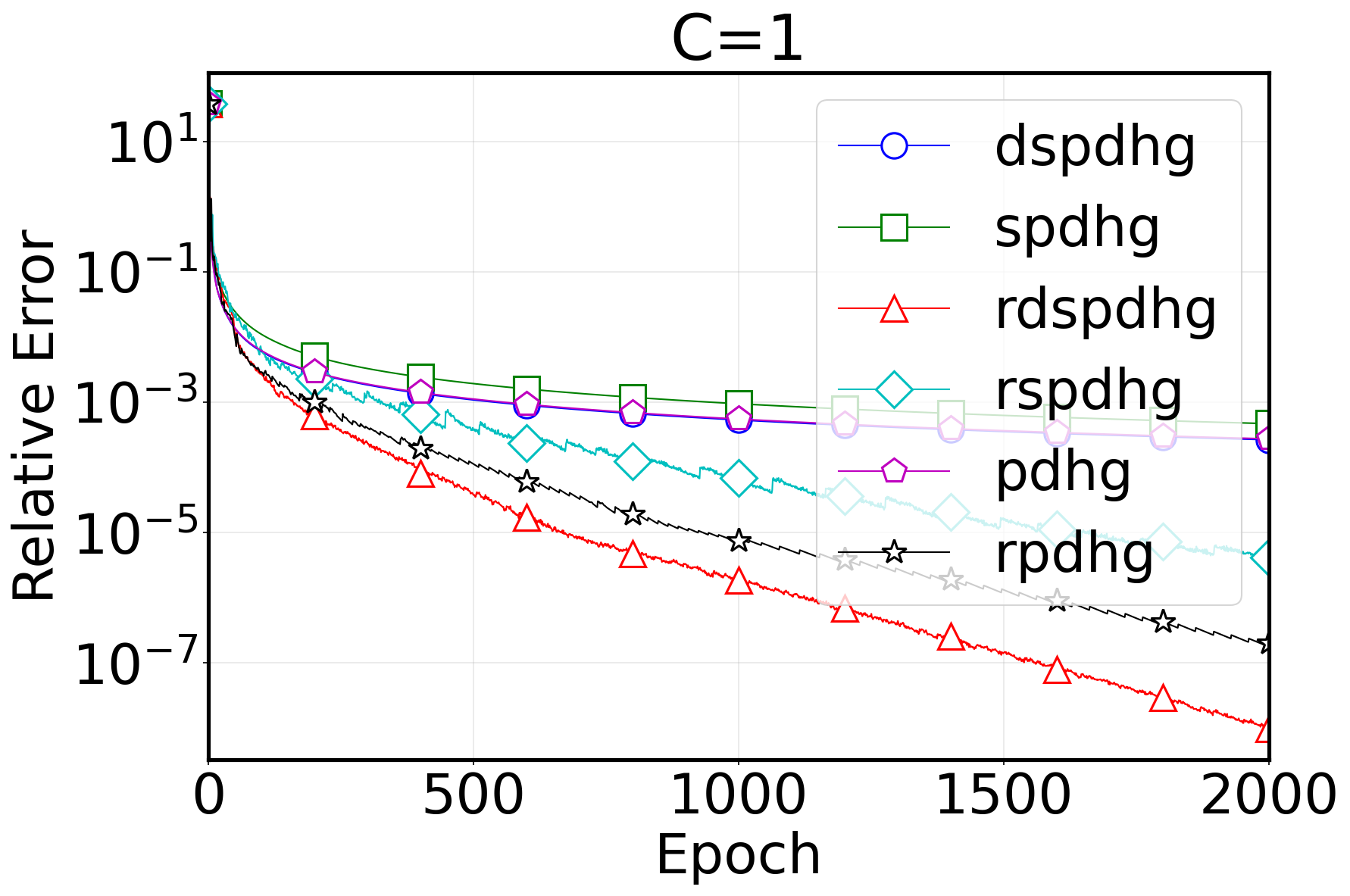}
    \end{minipage}
\caption{Relative error for soft-margin SVM. First column: colon-cancer ($n=62, m=2000$) with  $p=0.3, q=0.15$; second column: mushrooms ($n=m=100$) with $p=q=0.1$; third column: w1a ($n=2477, m=300$) with $p=0.2, q=0.4$.}
    \label{fig:svm}
\end{figure}

We first consider the soft-margin support vector machine problem
\begin{align}
    \min_{w,d}\quad \frac{1}{2}\|w\|^2
    + C\sum_{i=1}^n \max\bigl(0,\,1-b_i(w^\top a_i+d)\bigr),
    \label{eq:svm_problem}
\end{align}
where $w\in\mathbb{R}^m$, $d\in\mathbb{R}$, $a_i\in\mathbb{R}^m$ is the feature vector
of the $i$th sample, and $b_i\in\{-1,+1\}$ is the corresponding label. The constant
$C>0$ is the penalty parameter.

\paragraph{Reformulation}
Problem \eqref{eq:svm_problem} can be written in the form of
\eqref{eq:primal-problem} by introducing the variable $x=(w,d)\in\mathbb{R}^{m+1},$
and the matrix $A\in\mathbb{R}^{n\times (m+1)}$ whose $i$th row is $A_i=(b_i a_i^\top,\; b_i).$
We define
\[
f_i(z)=C\max(0,1-z),\; i=1,\dots,n, \quad \text{and} \quad g_j(x_j)=
\begin{cases}
\frac12 x_j^2, & j=1,\dots,m,\\
0, & j=m+1.
\end{cases}
\]
Then \eqref{eq:svm_problem} takes the block-separable form
\eqref{eq:primal-problem}. The Fenchel conjugate of $f_i$ and its proximal mapping are
\[
f_i^*(y)=
\begin{cases}
y, & y\in[-C,0],\\
+\infty, & \text{otherwise},
\end{cases} \quad \text{and} \quad \prox_{\sigma f_i^*}(x)=
\begin{cases}
-C, & x-\sigma<-C,\\
x-\sigma, & -C\le x-\sigma\le 0,\\
0, & x-\sigma>0.
\end{cases}
\]
Moreover,
\[
\prox_{\tau g_j}(x)=
\begin{cases}
\dfrac{x}{1+\tau}, & j=1,\dots,m,\\
x, & j=m+1.
\end{cases}
\]

\paragraph{Experimental setup} We use several benchmark datasets from LIBSVM~\cite{LibSVM}. The primal variable is
split coordinatewise, while the dual variable is naturally partitioned by samples. 
The penalty parameter $C$
and the dataset statistics are reported together with the corresponding convergence plots
in Figure~\ref{fig:svm}.

\subsection{Model predictive control}
We next consider the following model predictive control (MPC) problem:
\begin{align}
\begin{aligned}
\min_{\{x_t,u_t\}}\quad
& x_T^\top H_T x_T + \sum_{t=0}^{T-1}\bigl(x_t^\top H_t x_t + u_t^\top R_t u_t\bigr) \\
\text{s.t.}\quad
& x_{t+1}=Ax_t+Bu_t,\qquad t=0,1,\cdots, T-1,\\
& x_{t+1}\in\mathcal X,\quad u_t\in\mathcal U,\qquad t=0,1,\cdots, T-1,\\
& x_0=x_{\mathrm{init}},
\end{aligned}
\label{eq:oc_problem}
\end{align}
where $x_t\in\mathbb{R}^{n_x}$ and $u_t\in\mathbb{R}^{n_u}$ denote the state and control
variables, respectively. The box constraints are given by
\[
\mathcal X=\{x\in\mathbb{R}^{n_x}:-\bar x\le x\le \bar x\},
\qquad
\mathcal U=\{u\in\mathbb{R}^{n_u}:-\bar u\le u\le \bar u\}.
\]
The horizon length is $T$, and $x_{\mathrm{init}}\in\mathbb{R}^{n_x}$ is the initial state.
The matrices $H_t\in\mathbb{S}_+^{n_x}$, $R_t\in\mathbb{S}_+^{n_u}$, and
$H_T\in\mathbb{S}_+^{n_x}$ define the stage and terminal costs.

\paragraph{Reformulation}
Introducing dual variables $y_t$ for the dynamics constraints $x_{t+1}-Ax_t-Bu_t=0,\, t=0,\dots,T-1, $
yields the saddle-point formulation associated with the Lagrangian
\[
\min_{x_t\in\mathcal X,\;u_t\in\mathcal U}
\quad
x_T^\top H_T x_T
+\sum_{t=0}^{T-1}\bigl(x_t^\top H_t x_t+u_t^\top R_t u_t\bigr)
+\sum_{t=0}^{T-1} y_t^\top(x_{t+1}-Ax_t-Bu_t).
\]
This problem is naturally block structured across time stages, making it well suited
for stochastic primal--dual block updates.

\paragraph{Experimental setup}
We generate synthetic instances with $n_x=n_u=20$. The system matrix is chosen as $A=0.5I+\Delta, \; \Delta_{ij}\sim\mathcal N(0,0.1^2),$
and we keep only stable realizations whose eigenvalues lie strictly inside the unit circle.
The input matrix is generated from $B_{ij}\sim\mathcal N(0,1).$
The state cost matrix is $H_0 = \cdots = H_T=\operatorname{diag}(h)$, where $h_i\sim\mathcal U(0,10)$
and $70\%$ of the entries of $h$ are nonzero, and  $R_0 = \cdots = R_{T-1} =0.1I$.
The box constraints are generated from
$\bar x_i\sim\mathcal U(1,2),\,
\bar u_i\sim\mathcal U(0,0.1),$
and the initial state is sampled as $x_{\mathrm{init}}\sim \mathcal U(-0.5\bar x,\,0.5\bar x).$
We report results for horizons $T=20,50,100$. The resulting
convergence curves are reported in Figure~\ref{fig:oc}.

\begin{figure}[H]
    \centering
    \begin{subfigure}[b]{0.32\textwidth}
        \centering
        \includegraphics[width=\textwidth]{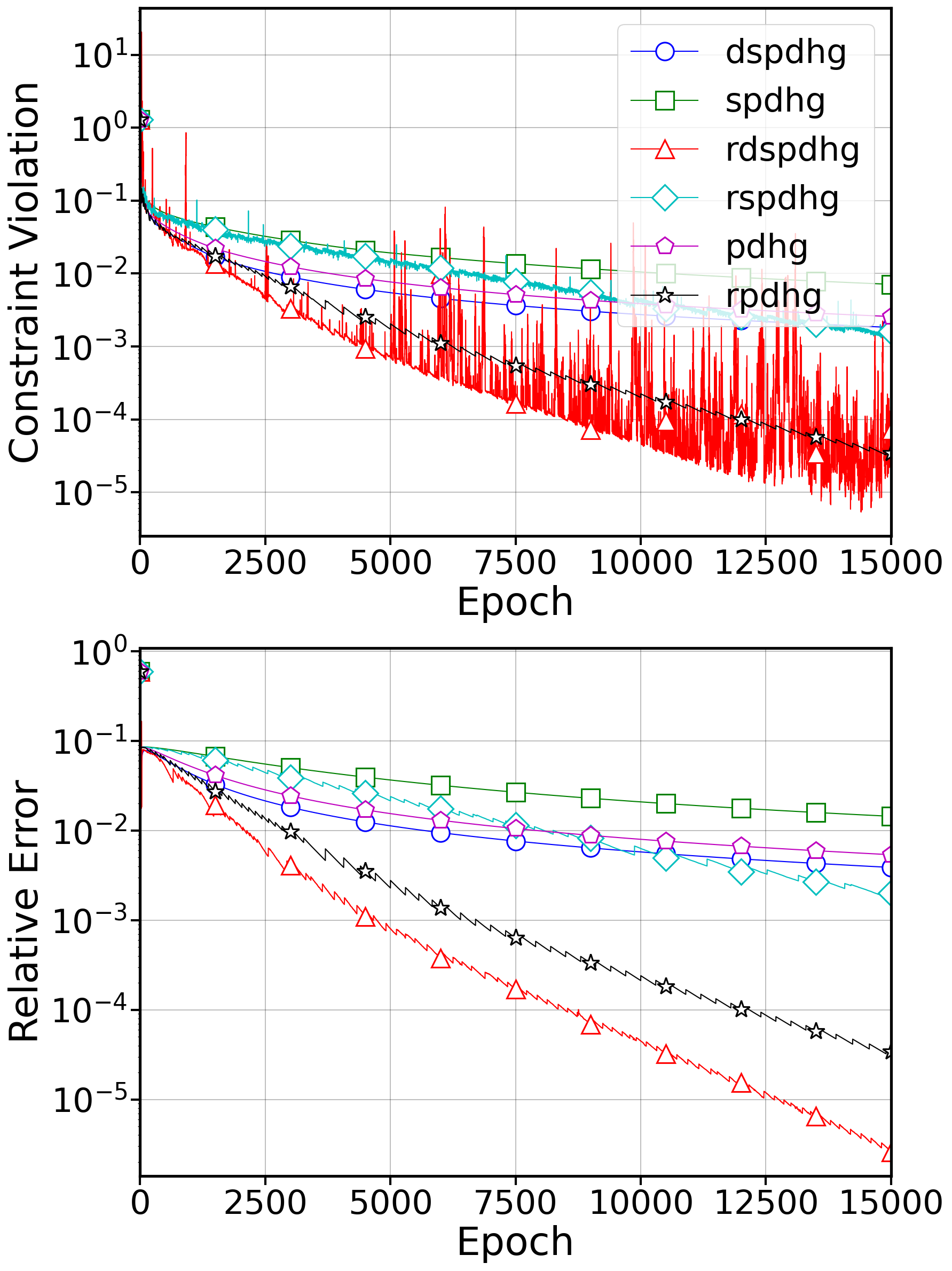}
        \label{fig:T20}
    \end{subfigure}
    \hfill 
    \begin{subfigure}[b]{0.32\textwidth}
        \centering
        \includegraphics[width=\textwidth]{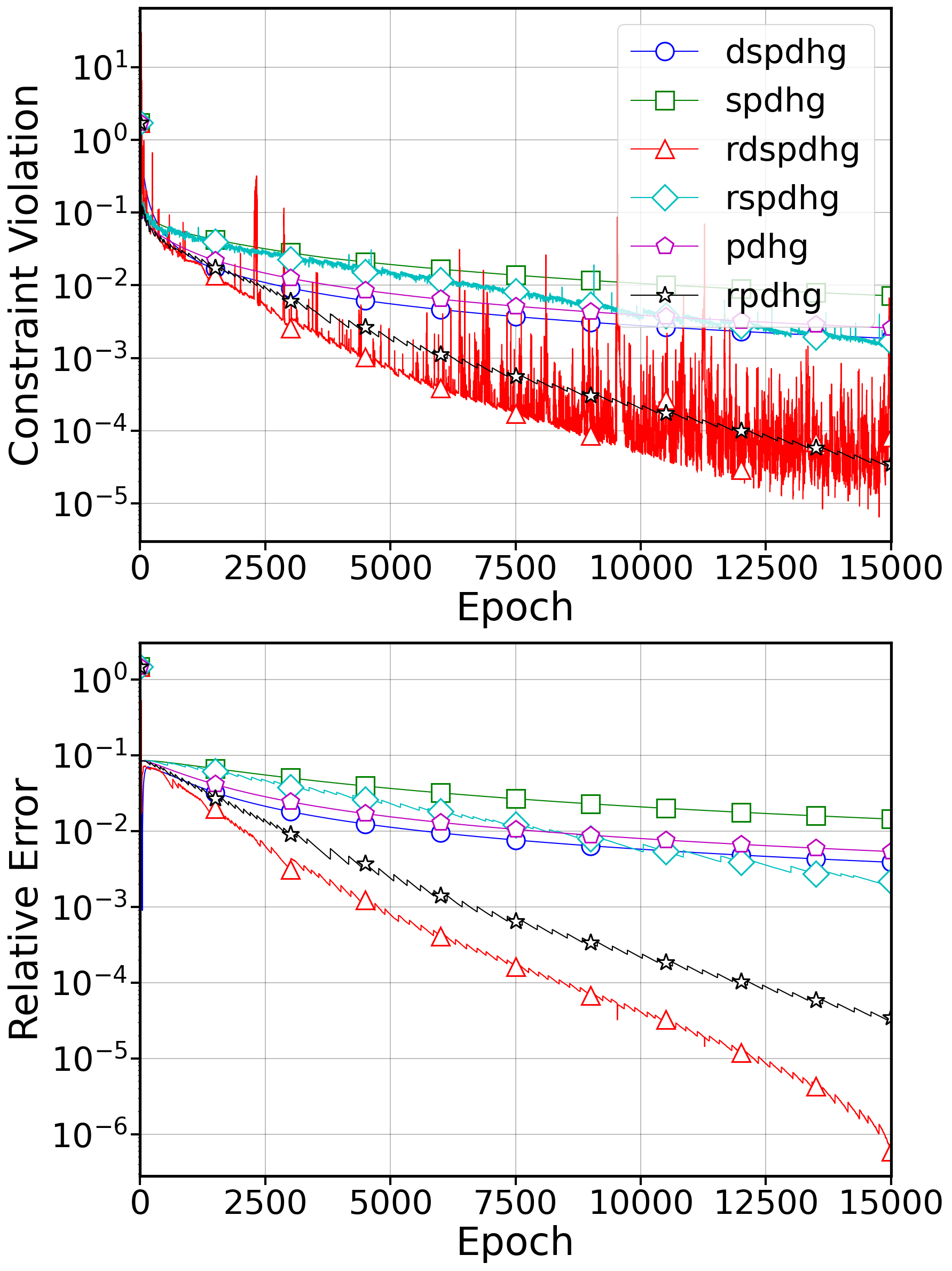}
        \label{fig:T50}
    \end{subfigure}
    \hfill
    \begin{subfigure}[b]{0.32\textwidth}
        \centering
        \includegraphics[width=\textwidth]{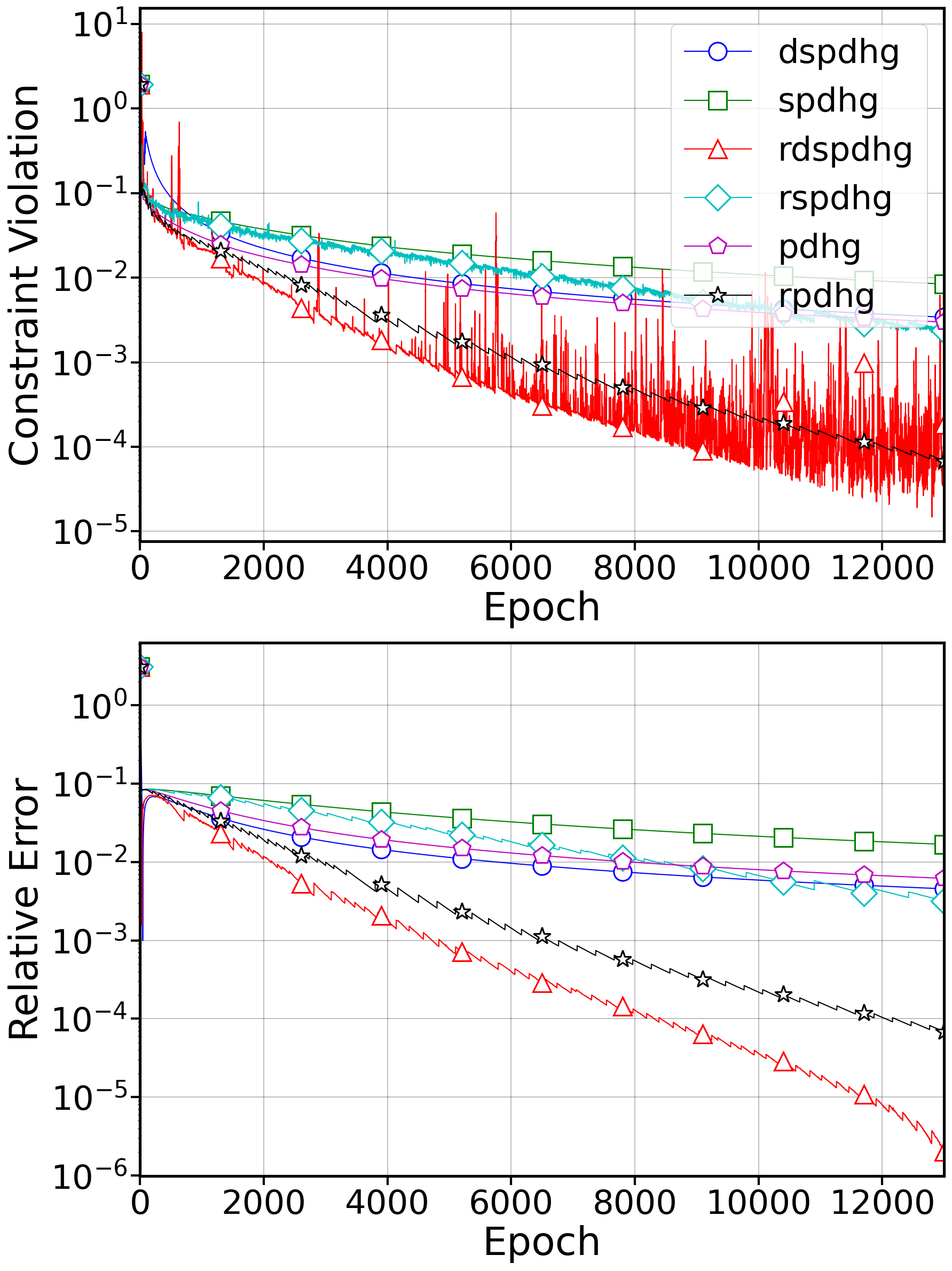}
        \label{fig:T100}
    \end{subfigure}
    \caption{Constraint violation and relative error for MPC problems with $n_x = n_u = 20$ and $p=0.1, q = 0.5$. First column: $T=20$; second column: $T=50$; third column: $T=100$.}
    \label{fig:oc}
\end{figure}

\paragraph{Overall observations}
As shown in both figures, the methods without restarting exhibit the expected sublinear convergence behavior, while restarting can significantly accelerate convergence when the local problem geometry is favorable. This is fully consistent with the theoretical results established in the previous sections. More importantly, across all tested instances, restarted DSPDHG delivers the best overall performance among the methods under comparison.

\section{Conclusion}

In this paper, we proposed DSPDHG, a doubly stochastic primal dual hybrid gradient
method for block-structured convex saddle-point problems, together with its restarted
variant RDSPDHG. By allowing randomized updates on both the primal and dual sides,
the method extends classical PDHG and one-sided stochastic PDHG in a unified manner.
For general convex problems, we established an $\mathcal{O}(1/K)$ ergodic convergence
rate for the expected restricted primal--dual gap. Under a smoothed quadratic growth
condition, we further proved linear convergence of the restarted scheme. Numerical
experiments on soft-margin SVM and model predictive control demonstrate that the
proposed methods are practically effective, and that restarting can substantially improve
performance. These results suggest that doubly stochastic primal--dual updates provide
a promising approach for solving large-scale structured convex optimization problems.

\bibliographystyle{plain}
\bibliography{references}






\end{document}